\documentclass[12pt]{amsart}
\usepackage{amssymb}
\usepackage{amsmath}
\usepackage{amsthm}
\usepackage{mathtools}
\usepackage{calrsfs}
\usepackage{appendix}
\usepackage[all]{xy}
\usepackage{url}
\usepackage[mathscr]{euscript}

\usepackage{tikz-cd}
\usepackage{hyperref}

\newtheorem{theorem}{Theorem}[section]

\newtheorem{proposition}[theorem]{Proposition}
\newtheorem{corollary}[theorem]{Corollary}

\theoremstyle{definition}
\newtheorem{definition}[theorem]{Definition}
\newtheorem{example}[theorem]{Example}
\newtheorem{remark}[theorem]{Remark}

\theoremstyle{remark}

\newcommand\bovermat[2]{%
	\makebox[-2pt][l]{$\smash{\overbrace{\phantom{%
					\begin{bmatrix}#2\end{bmatrix}}}^{\text{#1}}}$}#2}

\newcommand{\ZZ}{\mathbb{Z}}

\newcommand{\GG}{\mathbb{G}}

\newcommand{\CC}{\mathbb{C}}

\newcommand{\bA}{\mathbf{A}}

\newcommand{\bff}{\mathbf{f}}

\newcommand{\bK}{\mathbf{K}}

\newcommand{\rd}{\mathrm{d}}

\newcommand{\rP}{\mathrm{P}}

\newcommand{\bsy}{\boldsymbol{y}}
\newcommand{\bsz}{\boldsymbol{z}}

\DeclareMathOperator{\GL}{GL}
\DeclareMathOperator{\Mat}{Mat}

\DeclareMathOperator{\Frob}{Frob}
\DeclareMathOperator{\End}{End}
\DeclareMathOperator{\proj}{proj}

\DeclareMathOperator{\tens}{tens}

\DeclareMathOperator{\Gal}{Gal}
\DeclareMathOperator{\Hom}{Hom}

\DeclareMathOperator{\Exp}{Exp}
\DeclareMathOperator{\Id}{Id}
\DeclareMathOperator{\ch}{ch}

\DeclareMathOperator{\Lie}{Lie}

\newcommand{\oK}{\mkern2.5mu\overline{\mkern-2.5mu K}}
\newcommand{\oL}{\overline{L}}

\newcommand{\sep}{\mathrm{sep}}

\newcommand{\tr}{\mathrm{tr}}

\newcommand{\inorm}[1]{{\lvert #1 \rvert}}

\definecolor{chan}{rgb}{0.0, 0.5, 0.0}

\definecolor{oguz}{rgb}{0.0, 0.5, 0.0}

\begin{document}
		\title[The transcendence of special values of Goss $L$-functions]{On the transcendence of special values of Goss $L$-functions attached to Drinfeld modules}
	
	\author{O\u{g}uz Gezm\.{i}\c{s}}
\address{Department of Mathematics, National Tsing Hua University, Hsinchu City 30042, Taiwan R.O.C.}
\email{gezmis@math.nthu.edu.tw}

	\author{Changningphaabi Namoijam}
	\address{Department of Mathematics, Fairfield University, 1073 North Benson Road
Fairfield, CT 06824, U.S.A}
	\email{cnamoijam@fairfield.edu}
	

	
	\date{\today}


	\begin{abstract} Let $\mathbb{F}_q$ be the finite field with $q$ elements and consider the rational function field $K:=\mathbb{F}_q(\theta)$. For a Drinfeld module $\phi$ defined over $K$, we study the transcendence of special values of the Goss $L$-function attached to the abelian $t$-motive $M_{\phi}$ of $\phi$. Moreover, when $\phi$ is a Drinfeld module of rank $r\geq 2$ defined over $K$ which has everywhere good reduction, we prove that the value of the Goss $L$-function attached to the $(r-1)$-st exterior power of $M_{\phi}$ at any positive integer is transcendental over $K$.
	\end{abstract}

 	\maketitle
	\section{Introduction}
	\subsection{Motivation and Background}
	After Grothendieck  discovered the notion of motives in the 1960s, the $L$-functions attached to them have gotten enormous attention over the following years. These are the objects which can be seen as generalizations of several well-known functions such as the Riemann zeta function $\zeta(\mathfrak{s})$ defined for $\mathfrak{s}$ with $\Re(\mathfrak{s})>1$. In particular, setting $\mathbb{Q}(-1)$ to be the Lefschetz motive, the motivic $L$-function $L(\mathbb{Q}(-1),\cdot)$ of $\mathbb{Q}(-1)$ has the property that
	\begin{equation}\label{E:zet1}
	L(\mathbb{Q}(-1),s)=\zeta(s-1)=\sum_{m=1}^{\infty}\frac{1}{m^{s-1}}, \ \ s=3,4,\dots.
	\end{equation}
	
	One interesting phenomenon concerning these $L$-functions is to understand the transcendence of their special values. When $s$ is a positive  integer, it is known that $\zeta(2s)$ is transcendental over $\mathbb{Q}$ and the same result is  expected to hold for $\zeta(s)$ where $s\in \ZZ_{\geq 2}$. Regarding more general motives, Deligne \cite{Del79} conjectured that for a particular integer $s$ and a pure motive $\mathsf{M}$, $L(\mathsf{M},s)$ is equal to the product of a rational number and the determinant of a matrix consisting of periods arising from the comparison isomorphism between Betti and de Rham cohomologies of $\mathsf{M}$. Later on, generalizing the ideas of Deligne, Beilinson \cite{Bel85} conjectured that for a pure motive $\mathsf{M}$ with weight less than -2, $L(\mathsf{M},0)$ can be written, modulo $\mathbb{Q}^{\times}$, as the determinant of the Beilinson regulator (see also \cite{BK07} for a refinement of Belilinson's conjecture). We refer the reader to \cite{Nek94} and \cite{RSS88} for more details and the recent progress on the aforementioned conjectures.
	
	In this paper, we focus on the motivic $L$-functions which are defined over a field of prime characteristic by following a close analogy with the classical setting and study the transcendence of their special values. 
	
	In what follows, we introduce a few notation necessary to define our main objects. Let $p$ be a prime and $\mathbb{F}_q$ be the finite field with $q=p^m$ elements where $m\in \ZZ_{\geq 1}$. Considering a variable $\theta$ over $\mathbb{F}_q$, we let $A:=\mathbb{F}_q[\theta]$, the set of polynomials in $\theta$ with coefficients in $\mathbb{F}_q$ and $A_{+}$ be the set of monic polynomials in $A$. We further set $K$ to be the fraction field of $A$. We let $\inorm{\cdot}$ be the norm corresponding to the infinite place normalized so that $\inorm{\theta}=q$. Consider the formal Laurent series ring $\mathbb{F}_q((1/{\theta}))$ which is the completion of $K$ with respect to $\inorm{\cdot}$ and denote it by $K_{\infty}$. We let $\CC_{\infty}$ be the completion of a fixed algebraic closure of $K_{\infty}$ and set $\oK$ to be the algebraic closure of $K$ in $\CC_{\infty}$.
	
	Let $R$ be a subring of $\CC_{\infty}$ containing $A$. We define the non-commutative polynomial ring $R[\tau]$ subject to the condition 
	\[
	\tau c=c^q\tau,\ \ c\in R.
	\]
	
	Let $t$ be a variable over $\CC_{\infty}$ and set $\bA:=\mathbb{F}_q[t]$.  \textit{A Drinfeld $\bA$-module $\phi$ of rank $r\in \ZZ_{\geq 1}$ defined over $R$} is an $\mathbb{F}_q$-linear ring homomorphism $\phi:\bA\to R[\tau]$ which is uniquely given by 
	\begin{equation}\label{E:drinfeldgen}
	\phi_t:=\theta+c_1\tau+\dots+c_r\tau^r, \ \ c_r\neq 0.
	\end{equation}
	We call two Drinfeld $\bA$-modules $\phi$ and $\phi'$ \textit{isomorphic over $R$} if there exists an element $u\in R^{\times}$ so that $\phi_t u=u\phi'_t$. We further say that a Drinfeld $\bA$-module $\phi$ of rank $r$ defined over $K$ has \textit{everywhere good reduction} if $\phi$ is isomorphic, over $K$, to a Drinfeld $\bA$-module $\psi$ given by
	\begin{equation}\label{E:drinfeld}
	\psi_{t}=\theta+a_1\tau+\dots+a_r\tau^r
	\end{equation}
	so that $a_1,\dots,a_{r-1}\in A$ and $a_r\in \mathbb{F}_q^{\times}$ (see \cite[Sec. 1]{Bjo98}). As an example, \textit{the Carlitz module $C$} given by 
	\begin{equation}\label{E:C}
	C_{t}=\theta+\tau
	\end{equation}
	is a Drinfeld $\bA$-module of rank $1$ defined over $A$ and has everywhere good reduction. In addition to its tremendous importance for the present work, one can also see \cite{C35}, \cite{C38}, and \cite{H74} for its role to study class field theory for global function fields.
	
	\subsection{Special values of Goss \texorpdfstring{$L$}{L}-functions} In  \cite{And86}, Anderson introduced a function field analogue of motives in the classical setting which form a category closed under taking tensor products and direct sums (see \cite[Sec. 7.1]{Thakur}). We emphasize that Anderson called such objects ``$t$-motives'' but we will instead follow Goss's terminology \cite[Def. 5.4.12]{Goss}  and call them ``abelian $t$-motives'' in the rest of the paper. For any given abelian $t$-motive, Anderson further showed that one can associate, unique up to isomorphism, an abelian $t$-module which can be seen as a higher dimensional generalization of Drinfeld $\bA$-modules. In particular, to make the reader more familiar with the objects in use, we emphasize that any Drinfeld $\bA$-module is a one dimensional abelian $t$-module (see \S2.1 for more details).
	
	For the abelian $t$-motive $M_{\phi}$ corresponding to $\phi$ and $s\in \ZZ_{\geq 1}$, Goss \cite{Goss2}, closely following ideas of Gekeler \cite[Rem. 5.10]{Gek91}, defined the motivic $L$-function $L(M_{\phi},s)$ as well as the $L$-function $L(\wedge_{K[t]}^{k}M_{\phi},s)$ associated to the $k$-th exterior power of $M_{\phi}$ (see \S3.1 for more details and the explicit construction of these $L$-functions). Let us analyze an example of such $L$-functions in what follows. When $\phi=C$, the Carlitz module, one can have
	\begin{equation}\label{E:zet2}
	L(M_C,n)=\sum_{a\in A_{+}}\frac{1}{a^{n-1}}\in K_{\infty}, \ \ n\in \ZZ_{\geq 2}.
	\end{equation}
	These values are also known as \textit{Carlitz zeta values at $n-1$} whose construction dates back to Carlitz. One can also note, by comparing \eqref{E:zet1} with \eqref{E:zet2}, the immediate analogy of $L(M_C,n)$ with the Riemann zeta values at positive integers as well as the analogy between $M_C$ and $\mathbb{Q}(-1)$. Due to the rich transcendence theory established in the function field setting (see for example \cite{CP12,Pap08,Yu91}), many remarkable results on the special values of $L(M_{C},n)$  have been obtained. Just to name a few, when  $n-1$ is divisible by $q-1$, Wade \cite{Wad41} was able to prove that $L(M_C,n)$ is transcendental over $\oK$. In 1990, Anderson and Thakur were able to write $L(M_{C},n)$ as an $A$-linear combination of polylogarithms \cite[Thm. 3.8.3]{AndThak90}. Later on, building on the work of Anderson and Thakur and proving a version of Hermite-Lindemann theorem for tensor powers of the Carlitz module, Yu \cite[Thm. 3.1]{Yu91} showed that $L(M_{C},n)$ is transcendental over $\oK$ for any $n\in \ZZ_{\geq 2}$. 
	
	When $C$ is replaced by any Drinfeld $\bA$-module $\phi$  of arbitrary rank defined over $K$, explicit formulas for $L(M_{\phi},n)$ are not known yet except for some special cases studied in \cite{G20}. Nonetheless, using the methods developed in the present paper, we are able to conclude the transcendence of special values over $K$.

	\begin{theorem}\label{T:mainThm} Let  $n$ be a positive integer and $\phi$ be a Drinfeld $\bA$-module of rank $r$ defined over $K$. Then the following statements hold.
		\begin{itemize}
			\item[(i)] The special value $L(M_{\phi},n)$ is transcendental over $\oK$. 
			\item[(ii)] Assume that $\phi$ has everywhere good reduction and has rank $r\geq 2$. Then, $L(\wedge_{K[t]}^{r-1}M_{\phi},n)$ is transcendental over $\oK$.
		\end{itemize}
	\end{theorem}

\begin{remark}Our strategy for the proof of Theorem \ref{T:mainThm} relies on showing the transcendence of particular families of Taelman $L$-values (see \S3.4 for their definition). Hence we also provide an affirmative answer to \cite[Problem 4.1]{ADTR2020} for some certain cases (Corollary \ref{C:Taelman}).
	\end{remark}
	
	\subsection{Outline of the paper}
	In what follows, we describe the outline of the paper and briefly explain the method of our proof for the main result. In \S2, we discuss $t$-modules, provide explicit examples of them and introduce Taelman $t$-motives. In \S3, we analyze Goss $L$-functions attached to abelian $t$-motives and some explicitly determined Taelman $t$-motives as well as their relation with Taelman $L$-values. Studying Taelman $t$-motives rather than abelian $t$-motives enables us to define the notion of duality and relate them to special values of Goss $L$-functions (see \S3.1). In \S3.2, we define the abelian $t$-modules $\mathcal{E}_{n,A}$ and $G_{n,A}$ corresponding to certain abelian $t$-motives we introduce. Analyzing local factors of Goss $L$-functions as well as using Gardeyn's work in \cite[Sec. 8.4]{G01}, we obtain \eqref{E:comp1} and  \eqref{E:comp2} which enable us to interpret Theorem \ref{T:mainThm} in terms of Taelman $L$-values in Theorem \ref{T:3} and Theorem \ref{T:2} (see also Remark \ref{R:Carlitz} for the case of Drinfeld $\bA$-modules of rank $1$). Finally in \S4, using the work of Angl\`{e}s, Ngo Dac and Tavares Ribeiro \cite{ADTR} on Taelman's conjecture \cite[Conj. 1]{Taelman3}, we realize our Taelman $L$-values as a determinant of a matrix consisting of certain coordinates of logarithms of $t$-modules (Theorem \ref{T:ThmANDTR}). Then, we prove Theorem \ref{T:3} (Theorem \ref{T:2} resp.) by using the previous results of authors established in \cite[Thm. 1.1 and 1.2]{GN24} and Fang's result \cite[Thm. 1.10]{Fang} generalizing Taelman's class number formula \cite[Thm. 1]{Taelman}.	
	\subsection*{Acknowledgments}  The authors would like to thank Chieh-Yu Chang and Yen-Tsung Chen for fruitful discussions and valuable suggestions. The authors are also grateful to Shih-Chieh Liao for their presentation on Tate modules which shed light for us to improve the results in the previous version of the paper. The authors would also like to thank the referee for the careful reading of the manuscript and for the valuable comments.
    
    The first author was supported by NSTC Grant 113-2115-M-007-001-MY3. 
 The first author also acknowledges partial support by Deutsche Forschungsgemeinschaft (DFG) through CRC-TR 326 `Geometry and Arithmetic of Uniformized Structures,' project number 444845124 and National Center for Theoretical Sciences. The second author was partially supported by MOST Grant 110-2811-M-007-517.

	\section{\texorpdfstring{$t$}{t}-modules and Taelman \texorpdfstring{$t$}{t}-motives}
	In this section, we first review the notion of $t$-modules introduced by Anderson \cite{And86}. We also review effective $t$-motives which are closely related to $t$-modules. Then, we briefly explain Taelman $t$-motives and their fundamental properties.

	\subsection{\texorpdfstring{$t$}{t}-modules} Let us assume that $L$ is a field extension of $\mathbb{F}_q$. We say that $L$ is \textit{an $\bA$-field} if there exists a ring homomorphism $\ch: \bA\to L$. Throughout the paper, we consider any extension $L$ of $K$ in $\CC_{\infty}$ as an $\bA$-field where $\ch$ sends $t\mapsto \theta$. In this case, we say that $L$ has \textit{generic characteristic}. We note that any field extension of $K$ in $\mathbb{C}_{\infty}$ has generic characteristic.

	Let $m,\ell\in \ZZ_{\geq 1}$. For any matrix  $B=(b_{ik})\in \Mat_{m\times \ell}(L)$ and $j\in \mathbb{Z}$, set $B^{(j)}:=(b_{ik}^{q^j})\in  \Mat_{m\times \ell}(L)$. We also let  $\Mat_{m\times \ell}(L)[\tau]$ be the set of polynomials of $\tau$ with coefficients in $\Mat_{m\times \ell}(L)$. Moreover, when $m=\ell$, we define the ring $\Mat_{m}(L)[[\tau]]$ of power series of $\tau$ with coefficients in $\Mat_{m}(L):=\Mat_{m\times m}(L)$ subject to the condition 
	\[
	\tau B=B^{(1)}\tau, \ \ B\in \Mat_{m}(L).
	\]
	We also consider $\Mat_{m}(L)[\tau]\subset \Mat_{m}(L)[[\tau]]$, the subring of polynomials in $\tau$. Furthermore, for any $B=B_0+B_1\tau+\dots +B_k\tau^k\in \Mat_{m\times \ell}(L)[\tau]$, we set $\rd B:=B_0\in \Mat_{m\times \ell}(L)$. 
	\begin{definition} \label{D:def}
		\begin{itemize}
			\item[(i)] Let $L$ be an $\bA$-field. \textit{A $t$-module of dimension $s\in \ZZ_{\geq 1}$} is a tuple $G=(\mathbb{G}_{a/L}^{s},\varphi)$ of the $s$-dimensional additive algebraic group $\mathbb{G}_{a/L}^s$ over $L$ and an $\mathbb{F}_q$-linear ring  homomorphism $\varphi:\bA\to \Mat_s(L)[\tau]$ given by 
			\begin{equation}\label{E:varphi}
			\varphi(t):=A_0+A_1\tau+\dots+A_m\tau^m
			\end{equation}
			for some $m\in \ZZ_{\geq 0}$ so that $\rd_{\varphi}(t):=\rd \varphi(t)=A_0=\ch(t) \Id_s+N$ where $\Id_s$ is the $s\times s$ identity matrix and $N$ is a nilpotent matrix. 
			\item[(ii)] For each $0\leq i \leq m$, if all the entries of $A_i$ lie in an $\bA$-algebra $R$, then we say $G$ is \textit{defined over $R$}.
			\item[(iii)] Let $G_1=(\mathbb{G}_{a/L}^{s_1},\varphi_1)$ and $G_2=(\mathbb{G}_{a/L}^{s_2},\varphi_2)$ be two $t$-modules over a field $K\subseteq L\subseteq \mathbb{C}_{\infty}$.  \textit{A morphism $\rP: G_1 \rightarrow G_2$ defined over $L$} is an element $\rP \in \Mat_{s_2 \times s_1}(L[\tau])$ satisfying
			\begin{equation*}\label{E:isogn}
			\rP\varphi_1(t) = \varphi_2(t) \rP.
			\end{equation*}
			We set $\Hom_L(G_1,G_2)$ to be the set of morphisms $G_1\to G_2$ defined over $L$. Moreover, when $s_1=s_2$, we call $G_1$ and $G_2$  \textit{isomorphic over $L$} if there exists an element $u\in \GL_{s_1}(L)$ so that $u\varphi_1(t) = \varphi_2(t) u.$ Let $\End(G):=\Hom_{\overline{K}}(G,G)$ be the ring of endomorphisms of $G=(\mathbb{G}_{a/{L}}^{s},\varphi)$ over $\overline{K}$. Observe that $\End(G)$ has an $\bA$-module structure  given by 
			\[
			a\cdot \rP:=\varphi(a)\rP, \ \ a\in \bA,  \ \ \rP\in \End(G).
			\] 
		\end{itemize}
	\end{definition}
	
	\begin{example}\label{Ex:0}  For any $n\in \ZZ_{\geq 1}$ and $\mathfrak{b}\in K\setminus \{0\}$, we define  $C_{\mathfrak{b}}^{\otimes n}:=(\mathbb{G}_{a/K}^n,\rho)$ where the $\mathbb{F}_q$-linear homomorphism  $\rho:\bA\to \Mat_n(K)[\tau]$ is given by 
		\[
		\rho(t)=\begin{bmatrix}
		\theta&1& & \\
		& \ddots&\ddots & \\
		& & \theta & 1\\
		& & & \theta  
		\end{bmatrix}+\begin{bmatrix}
		0&\dots&\dots &0 \\
		\vdots& & &\vdots \\
		0& &  & \vdots\\
		\mathfrak{b}&0&\dots& 0 
		\end{bmatrix}\tau.
		\] 
        When $n=1$, we simply set $C_{\mathfrak{b}}:=C_{\mathfrak{b}}^{\otimes 1}$. Note that when $n=1$ and $\mathfrak{b}=1$, we obtain the Carlitz module $C$ defined in \eqref{E:C}. Moreover, when $\mathfrak{b}=1$, we call $C^{\otimes n}:=C_1^{\otimes n}$  \textit{the $n$-th tensor power of the Carlitz module} (see \cite{AndThak90} for more details on $C^{\otimes n}$).
	\end{example}	
	For any $\bA$-algebra $R\subset L$ and a $t$-module $G=(\mathbb{G}_{a/L}^s,\varphi)$ defined over $R$, we set $\Lie(G)(R):=\Mat_{s\times 1}(R)$ and equip it with the $\bA$-module structure given by 
	\[
	t \cdot x:=\rd_{\varphi}(t)x=A_0x, \ \ x\in \Lie(G)(R).
	\]
	Furthermore, we also define  $G(R):=\Mat_{s\times 1}(R)$ whose $\bA$-module structure is given by 
	\[
	t\cdot x:=\varphi(t)x=A_0x+A_1x^{(1)}+\dots+A_mx^{(m)}, \ \ x\in G(R).
	\]

	Let $L$ be an $\bA$-field which has generic characteristic. For any $t$-module $G$ defined over $L$, there exists a unique infinite series $\Exp_G:=\sum_{i\geq 0}\alpha_i\tau^i\in \Mat_s(L)[[\tau]]$ satisfying $\alpha_0=1$ and 
	\begin{equation}\label{E:FuncEqtn}
	\Exp_G\rd_{\varphi}(t)=\varphi(t)\Exp_G.
	\end{equation}
	It leads to \textit{the exponential function of $G$}, which is an everywhere convergent $\mathbb{F}_q$-linear homomorphism $\Exp_G:\Lie(G)(\CC_{\infty})\to G(\CC_{\infty})$ given by 
	$\Exp_G(x)=\sum_{i\geq 0}\alpha_ix^{(i)}$ for any $x\in \Lie(G)(\CC_{\infty})$.

	\subsection{Effective \texorpdfstring{$t$}{t}-motives and abelian \texorpdfstring{$t$}{t}-modules} In this subsection, we introduce effective $t$-motives defined by Taelman \cite{Taelman2}. We refer the reader to  \cite{Taelman2} and \cite{Taelman3} for further details on the subject.
	
	Let $L$ be an extension of $\mathbb{F}_q$ which is also an $\bA$-field. We define $L[t]$ to be the set of polynomials in the variable $t$ with coefficients in $L$ and $L(t)$ to be its fraction field. For any $f=\sum_{i\geq 0}a_it^i\in L[t]$ and $j\in \ZZ$, we set $f^{(j)}:=\sum_{i\geq 0}a_i^{q^j}t^i\in L[t]$. Furthermore, we set $L[t,\tau]:=L[t][\tau]$ to be the non-commutative ring defined subject to the condition
	\[
	\tau f =f^{(1)}\tau ,\ \ \ \  f\in L[t].
	\]
	
	Unless otherwise stated, the tensor products in this subsection are over $L[t]$.
	
	\begin{definition}
		\begin{itemize}
			\item[(i)] \textit{An effective $t$-motive $M$ over $L$} is a left  $L[t,\tau]$-module which is free and finitely generated over $L[t]$ with the property that the determinant of the matrix representing the $\tau$-action on $M$ is equal to $c(t-\ch(t))^n$ for some $c\in L^{\times}$ and $n\in \ZZ_{\geq 1}$. 
			\item[(ii)] The set $\mathfrak{M}_{L}$ of effective $t$-motives over $L$ forms a category where the morphisms between any two effective $t$-motives are left $L[t,\tau]$-module homomorphisms.  The set of morphisms between $M_1\in \mathfrak{M}_L$ and $M_2\in \mathfrak{M}_L$ is denoted by $\Hom_{\mathfrak{M}_L}(M_1,M_2)$.
			\item[(iii)] The tensor product $M_1\otimes M_2$ of effective $t$-motives $M_1$ and $M_2$ is an effective $t$-motive on which $\tau$ acts diagonally.
			\item[(iv)] For any $\alpha\in L^{\times}$, we define the effective $t$-motive $ \textbf{1}_{\alpha}:=L[t]$ whose $L[\tau]$-action is given by $\tau f=\alpha f^{(1)}$ for all $f\in L[t]$. When $\alpha=1$, we simply write $\textbf{1}:=\textbf{1}_{1}$.
			\item[(v)] An effective $t$-motive which is free and finitely generated over $L[\tau]$ is called \textit{an abelian $t$-motive over $L$}.
		\end{itemize}
	\end{definition}
	
	One can construct a $t$-module from an abelian $t$-motive $M$ over $L$ as follows (see \cite[\S4.1]{BP20} for further details): Let $\{v_1,\dots,v_s\}$ be an $L[\tau]$-basis for $M$. Then, there exists a matrix $\Phi_{\theta}\in \Mat_s(L)[\tau]$ such that 
	\[
	t\cdot \begin{bmatrix}
	v_1\\\vdots \\ v_s
	\end{bmatrix}=\Phi_{\theta}\begin{bmatrix}
	v_1\\\vdots \\ v_s
	\end{bmatrix}.
	\]
	Now considering the $\mathbb{F}_q$-linear homomorphism $\psi:\bA\to \Mat_s(L)[\tau]$ given by $\psi(t)=\Phi_{\theta}$, one sees that $(\mathbb{G}^s_{a/L},\psi)$ forms a $t$-module. A $t$-module constructed in this way is called \textit{an abelian $t$-module}. 
	
	We remark that there is an anti-equivalence between the category of abelian $t$-motives over $L$ and the category of abelian $t$-modules defined over $L$  (see \cite[Thm. 1]{Taelman3}). Although we will not go into the details, in what follows, we introduce several examples of this correspondence where $L=K$. 
	\begin{example}\label{Ex:1} 
		\begin{itemize}
			\item[(i)] Let $\phi$ be a Drinfeld $\bA$-module of rank $r$ defined over $K$ as in \eqref{E:drinfeldgen}. 
            Let $M_{\phi}:=\oplus_{i=1}^rK[t]m_i$ be the free $K[t]$-module with some chosen $K[t]$-basis $\{m_1,\dots,m_r\}$ and equip it with the left $K[\tau]$-module structure given by 
			\[
		c\tau\cdot\sum_{i=1}^rf_im_i=c(f_r^{(1)}(t-\theta)c_r^{-1}m_1+(f_1^{(1)}-f_r^{(1)}c_1c_r^{-1})m_2+\dots+ (f_{r-1}^{(1)}-f_r^{(1)}c_{r-1}c_r^{-1})m_r)
			\]
			where $f_1,\dots,f_r\in K[t]$ and $c\in K$. Then, $M_{\phi}$ is free of rank $1$ over $K[\tau]$ with $K[\tau]$-basis $\{m_1\}$ and hence is the abelian $t$-motive over $K$ corresponding to $\phi$. 
			\item[(ii)] For any $n\in \ZZ_{\geq 1}$, let $\textbf{C}_{\mathfrak{b}}^{\otimes n}:=K[t]m$ for some $K[t]$-basis $\{m\}$ so that $\tau \cdot fm=f^{(1)}\mathfrak{b}(t-\theta)^nm$ where $f\in K[t]$. It forms a left $K[t,\tau]$-module and an easy computation yields that $\textbf{C}_{\mathfrak{b}^{-1}}^{\otimes n}$ is the abelian $t$-motive over $K$ corresponding to $C_{\mathfrak{b}}^{\otimes n}$ defined in Example \ref{Ex:0}. For convenience, we further set $\textbf{C}_{\mathfrak{b}}:=\textbf{C}_{\mathfrak{b}}^{\otimes 1}$.
			\item[(iii)] For any effective $t$-motive $M$ over $K$ which has rank $r$  over $K[t]$ and for any $1\leq i \leq r$, we denote by $\wedge_{K[t]}^{i}M$ the $i$-th exterior power of $M$. The action of $\tau$ on $\wedge_{K[t]}^{i}M$ is induced from the action of $\tau$ on the $i$-th tensor power of $M$ and hence it can be realized as a left $K[t,\tau]$-module. Moreover, by \cite[Prop.~1]{Boc05}, we know that $\wedge_{K[t]}^{i}M$ is an effective $t$-motive. Furthermore, we set $\det(M)$ to be the \textit{$r$-th exterior power of $M$}. By construction, $\det(M)$ is a $K[t]$-module of rank $1$. If we consider $M=M_{\phi}$ as above, then we see that $\det(M_{\phi})$ is the effective $t$-motive over $K$ given by $\det(M_{\phi})=K[t](m_1\wedge \dots\wedge m_r)$ so that 
			\[ 
			\tau \big(f(m_1\wedge \dots\wedge m_r)\big)=f^{(1)}(-1)^{r-1}c_r^{-1}(t-\theta)(m_1\wedge \dots\wedge m_r),\ \ f\in K[t].
			\]
			This also implies that $\det(M_{\phi})$ is a free left $K[\tau]$-module with the basis $\{m_1\wedge \dots\wedge m_r\}$ and hence is an abelian $t$-motive over $K$.
		\end{itemize}	
	\end{example}

	We finish this subsection with the following definition due to Anderson \cite[Sec.~1.9]{And86} which will be used in \S 3.
	\begin{definition}
		\begin{itemize}
			\item[(i)] Assume that $L$ is an algebraically closed field and let  $L((1/t))$ be the ring of formal Laurent series. 
            Let $M$ be a left $L[t,\tau]$-module and consider the $L[t,\tau]$-module $M((1/t)):=M\otimes L((1/t))$ so that 
			\[
			\tau\cdot \Big(m\otimes \sum_{i=n_0}^{\infty}b_it^i\Big):=\tau m\otimes \sum_{i=n_0}^{\infty}b_i^qt^i.
			\]
			\item[(ii)] Let $L[[1/t]]$ be the ring of power series of $1/t$ with coefficients in $L$. A finitely generated $L[[1/t]]$-submodule $J$ of $M((1/t))$ is called \textit{a lattice} if $J\otimes_{L[[1/t]]}L((1/t))\cong M((1/t))$.
			\item[(iii)] We say that $M$ is \textit{pure} if it is free and finitely generated over $L[t]$ and there exists a lattice $J\subseteq  M((1/t))$ and $s_1,s_2\in \ZZ_{\geq 1}$ such that $\tau^{s_1}J=t^{s_2}J$.
			\item[(iv)] An abelian $t$-module $G$ is \textit{pure} if its corresponding abelian $t$-motive over $L$ is pure. 
		\end{itemize}
	\end{definition}
	
	\subsection{Taelman \texorpdfstring{$t$}{t}-motives} In this subsection, we assume that $L$ is a field of generic characteristic. For $M_1, M_2\in \mathfrak{M}_L$, we define $\Hom_{L[t]}(M_1,M_2)$ to be the set of $L[t]$-module homomorphisms between $M_1$ and $M_2$. Let $\oL$ be a fixed algebraic closure of $L$. One can define a left  $\oL[\tau]$-module action on the $\oL(t)$-module $\Hom_{\oL(t)}(M_1\otimes \oL(t),M_2\otimes \oL(t))$ by setting
	\[
	\tau f=\tau_2\circ f \circ \tau_1^{-1} 
	\]
	for all $f\in \Hom_{\oL(t)}(M_1\otimes \oL(t),M_2\otimes \oL(t))$ where $\tau_1^{-1}$ ($\tau_2$ respectively) is the semi-linear map $\tau^{-1}$ ($\tau$ respectively) on $M_1\otimes \oL(t)$ ($M_2\otimes \oL(t)$ respectively). Furthermore, one can obtain that (see \cite[Sec. 2.2.3]{Taelman2} for details) for a sufficiently large $n$, $\Hom_L(M_1,M_2\otimes C^{\otimes n})\subseteq \Hom_{\oL(t)}(M_1\otimes \oL(t),M_2\otimes C^{\otimes n}\otimes \oL(t)) $ has a left $L[\tau]$-module structure and hence it can be regarded as an effective $t$-motive over $L$.
	\begin{definition}
		\begin{itemize}
			\item[(i)] \textit{A Taelman $t$-motive $\mathbb{M}$ over $L$} is a tuple $\mathbb{M}=(M,n)$ consisting of an effective $t$-motive $M \in \mathfrak{M}_L$ and an $n\in \ZZ$.
			\item[(ii)] Let $\mathcal{M}_L$ be the set of all Taelman $t$-motives over $L$. For any  $\mathbb{M}_1=(M_1,n_1)$ and $\mathbb{M}_2=(M_2,n_2)$, the set of morphisms between $\mathbb{M}_1$ and $\mathbb{M}_2$ are given by 
			\[
			\Hom_{\mathcal{M}_L}(\mathbb{M}_1,\mathbb{M}_2):=\Hom_{\mathfrak{M}_L}(M_1\otimes C^{\otimes (n+n_1)},M_2\otimes C^{\otimes (n+n_2)})
			\]
			where $n\geq \max\{-n_1,-n_2\}$. We note that the definition is independent of the choice of $n$ (see \cite[Sec. 2.3]{Taelman2}).
			\item[(iii)] Continuing with the notation in (ii), we define the operation $\otimes_{\mathcal{M}_L}$ on $\mathcal{M}_{L}$ by 
			\[
			\mathbb{M}_1\otimes_{\mathcal{M}_L} \mathbb{M}_2:=(M_1\otimes M_2,n_1+n_2).
			\]
			\item[(iv)] The internal hom $\Hom_{L[t]}(\mathbb{M}_1,\mathbb{M}_2)$  is given by 
			\begin{equation}\label{E:inthom}
			\Hom_{L[t]}(\mathbb{M}_1,\mathbb{M}_2):=(\Hom_{L[t]}(M_1,M_2\otimes \textbf{C}^{\otimes (n_2-n_1+i)}),-i)
			\end{equation}
			for some integer $i$. Note that the definition is independent of the choice of $i$ when it is sufficiently large. 
		\end{itemize}
	\end{definition}
	For any $\mathbb{M}_1,\dots,\mathbb{M}_4\in \mathcal{M}_L$, one can obtain (see \cite[Sec.~2.3.5]{Taelman2}) that 
	\begin{equation}\label{E:dual2}
	\Hom_{L[t]}(\mathbb{M}_1,\mathbb{M}_2)\otimes_{\mathcal{M}_L} \Hom_{L[t]}(\mathbb{M}_3,\mathbb{M}_4)\cong \Hom_{L[t]}(\mathbb{M}_1\otimes \mathbb{M}_3,\mathbb{M}_2\otimes \mathbb{M}_4).
	\end{equation}

	\begin{remark}Let $M\in \mathfrak{M}_L$ be an effective $t$-motive. Using the fully faithful functor between the categories $\mathfrak{M}_L$ and $\mathcal{M}_L$ sending $M\mapsto (M,0)$, we denote the element $(M,0)$ by $M$ throughout the paper and it should not lead to any confusion.

	\end{remark}
	 For any given effective $t$-motive $M$, one can define \textit{the dual $M^{\vee}$ of $M$} to be $\Hom_{L[t]}(M,\textbf{1})$. We remark  that, if we assume that the $\tau$-action on $M$  with respect to a $K[t]$-basis is represented by $\Phi$  whose determinant is equal to $c(t-\theta)^{\ell}$ for some $\ell\in \ZZ_{\geq 1}$ and $c\in \mathbb{C}_{\infty}^{\times}$, then there exists a $K[t]$-basis of $\Hom_{K([t]}(M,\textbf{1})$ such that the $\tau$-action is represented by $(\Phi^{-1})^{\tr}$. Observe that in general $M^{\vee}$ is not an element in $\mathfrak{M}_L$. However, it can be always represented by an element $(M',i)$ in $\mathcal{M}_L$ for some effective $t$-motive $M'$ over $L$ and for some $i\in \ZZ$ (see Remark  \ref{R:dualmot} to see an example of this case). Furthermore, setting $\mathbb{M}_2=\mathbb{M}_3=\textbf{1}$ and $\mathbb{M}_4=\mathbb{M}_2$ in \eqref{E:dual2} yields
\begin{equation}\label{E:dualidentity}
    \Hom_{L[t]}(\mathbb{M}_1,\textbf{1})\otimes_{\mathcal{M}_L} \Hom_{L[t]}(\textbf{1},\mathbb{M}_2)\cong \mathbb{M}_1^{\vee}\otimes_{\mathcal{M}_L} \mathbb{M}_2\cong \Hom_{L[t]}(\mathbb{M}_1, \mathbb{M}_2).
\end{equation}

We finish this subsection by providing some examples.

\begin{example}\label{Ex:3}
			\begin{itemize}
				\item[(i)] Let $n\in \ZZ_{\geq 1}$. One can see that $\Hom_{K[t]}(\textbf{C}_{\mathfrak{b}}^{\otimes n},\textbf{C}^{\otimes n})\subseteq \Hom_{\oK(t)}(\textbf{C}_{\mathfrak{b}}^{\otimes n}\otimes \oK(t),\textbf{C}^{\otimes n}\otimes \oK(t))$ is isomorphic to $\textbf{1}_{\mathfrak{b}^{-1}}$ as left $K[t,\tau]$-modules. Furthermore, using \eqref{E:inthom}, we see that 
				\[
		(\textbf{C}_{\mathfrak{b}^{-1}}^{\otimes n})^{\vee}=\Hom_{K[t]}(\textbf{C}_{\mathfrak{b}^{-1}}^{\otimes n},\textbf{1})=(\Hom_{K[t]}(\textbf{C}_{\mathfrak{b}^{-1}}^{\otimes n},\textbf{C}^{\otimes n}),-n)=(\textbf{1}_{\mathfrak{b}},-n)=\textbf{C}_{\mathfrak{b}}\otimes (\textbf{C}^{\otimes (n+1)})^{\vee}
				\]
                which will be useful later in Remark \ref{R:Carlitz}.
				\item[(ii)] Let $M$ be an effective $t$-motive over $K$. The natural left $K[t,\tau]$-module isomorphism between $M\otimes \textbf{C}\otimes \textbf{C}^{\otimes n}$ and $M\otimes \textbf{C}^{\otimes (n+1)}$ implies that 
				\[
				(M\otimes \textbf{C}, i)\cong (M,i+1), \ \ i\in \ZZ.
				\]
				In particular, $\textbf{C}^{\otimes n}$ can  also be  represented as $(\textbf{1},n)$ in $\mathcal{M}_K$.
                \end{itemize}
                \end{example}

\subsection{Tensor product of \texorpdfstring{$M_{\phi}$}{Mphi} with \texorpdfstring{$\textbf{C}^{\otimes n}$}{Cn}}
For $n\in \ZZ_{\geq 1}$ and a Drinfeld $\bA$-module $\phi$ of rank $r$ defined over $K$ given as in \eqref{E:drinfeldgen}, we now consider the left $K[t,\tau]$-module $\textbf{G}_n:=M_{\phi}\otimes \textbf{C}^{\otimes n}$ on which $\tau$ acts diagonally. By Example \ref{Ex:1}, we see that $\textbf{G}_n$ has a $K[t]$-basis given by $\{m_1\otimes m,\dots,m_r\otimes m\}$. For $1\leq i \leq r$ and $0\leq j \leq n-1$, we set $v_{i,j}:=m_i\otimes (t-\theta)^{j}m$ and furthermore, we define $v_{1,n}:=m_1\otimes (t-\theta)^nm$. Using the left $K[\tau]$-module structure on $M_{\phi}$ and $\textbf{C}^{\otimes n}$, one can see that the set $\{v_{i,j}|\ \  1\leq i \leq r, \text{ } 0\leq j \leq n-1\}\cup \{v_{1,n}\}$ forms a left $K[\tau]$-module basis for $\textbf{G}_n$ and moreover, one can obtain
\begin{align*}
&(t-\theta)v_{i,j}=v_{i,j+1} \text{ for } 0\leq j \leq n-2, \text{ and } 1\leq i \leq r ,\\
&(t-\theta)v_{1,n-1}=v_{1,n},\\
&(t-\theta) v_{k,n-1}=\tau v_{k-1,0} \text{ for } 2\leq k \leq r,\\
&(t-\theta) v_{1,n}=c_1\tau v_{1,0}+\dots+c_r\tau v_{r,0}. 
\end{align*}
Hence, we see that 
\[
t \cdot \begin{bmatrix}
v_{1,0},
\dots,
v_{r,0},
\dots,
v_{1,n-1},
\dots,
v_{r,n-1},
v_{1,n}
\end{bmatrix}^{\tr}=\phi_{n}(t)\begin{bmatrix}
v_{1,0},
\dots,
v_{r,0},
\dots,
v_{1,n-1},
\dots,
v_{r,n-1},
v_{1,n}
\end{bmatrix}^{\tr}
\]
where $\phi_n:\bA\to \Mat_{rn+1}(K)[\tau]$ is the $\mathbb{F}_q$-algebra homomorphism  given by 
\[
\phi_n(t)=\theta \Id_{rn+1}+N+E\tau
\]
so that $N\in \Mat_{rn+1}(\mathbb{F}_q)$ and $E\in \Mat_{rn+1}(K) $ are  defined as 

\begin{equation}\label{E:matrices}
N:=\begin{bmatrix}
0&\cdots &0&\bovermat{$rn+1-r$}{1&0& \cdots & 0}  \\
& \ddots& & \ddots& \ddots & &\vdots \\
& &\ddots& &\ddots & \ddots&0\\
& &  & 0&\cdots&0 & 1\\
& &  & & 0&\cdots& 0\\
& &  & & &\ddots& \vdots\\
& &  & & & & 0\\
\end{bmatrix}\begin{aligned}
&\left.\begin{matrix}
\\
\\
\\
\\
\end{matrix} \right\} %
rn+1-r\\ %
&\left.\begin{matrix}
\\
\\
\\
\end{matrix}\right\}%
r\\
\end{aligned}
\end{equation}
and 
\begin{equation}\label{E:matricesA}
E:=\begin{bmatrix}
0&\cdots&\cdots & \cdots &\cdots & \cdots & 0\\
\vdots& & & & & &\vdots\\
0& & & & & & 0\\
1&0 &\cdots &\cdots &\cdots &\cdots &0\\
0&\ddots& \ddots&  & & &\vdots\\
\vdots& &1& \ddots & & & \vdots\\
c_1&\cdots&\cdots &c_r&0&\cdots& 0
\end{bmatrix}\begin{aligned}
&\left.\begin{matrix}
\\
\\
\\
\end{matrix} \right\} %
rn+1-r\\ %
&\left.\begin{matrix}
\\
\\
\\
\\
\end{matrix}\right\}%
r\\
\end{aligned}.
\end{equation}
Note that the last $r$-rows of $N$ contain only zeros. Thus, the abelian $t$-module corresponding to $\textbf{G}_n$ is given by $G_n:=(\mathbb{G}_{a/K}^{rn+1},\phi_n)$. We refer the reader to \cite{H93} for further details on the tensor products of Drinfeld $\bA$-modules of arbitrary rank.

\begin{remark} When $\phi=C$, the Carlitz module, $G_n$ is equal to $C^{\otimes (n+1)}$ defined as in Example \ref{Ex:0}.
\end{remark}

 To be consistent with the notation in this subsection, from now, we  set $G_0:=(\mathbb{G}_{a/K},\phi)$ to be the Drinfeld $\bA$-module (a $t$-module of dimension one) of rank $r$ given as in \eqref{E:drinfeldgen}.

    \subsection{Tensor product of \texorpdfstring{$(r-1)$}{r-1}-st exterior power of \texorpdfstring{$M_{\phi}$}{Mphi}  with \texorpdfstring{$\textbf{C}^{\otimes n}$}{Cn}} In this subsection, for any non-negative integer $n$, we define the  family $\widetilde{\mathcal{E}}_n$ of (abelian) $t$-modules studied in \cite[\S3]{GN24} corresponding to the abelian $t$-motive $\wedge_{K[t]}^{r-1}M_{\phi}\otimes \textbf{C}^{\otimes n}$ constructed by the tensor product of the $(r-1)$-st exterior power of $M_{\phi}$ with the $n$-th tensor power of the Carlitz module. 

    First, we need the following proposition.

	\begin{proposition}[{cf. \cite[Eqn. (1.2.1)]{Tag96}}]\label{P:prop1} Let $G_0= (\GG_{a/K}, \phi)$  be the Drinfeld $\bA$-module of rank $r\geq 2$ defined over $K$ given as in \eqref{E:drinfeldgen}. We have the following isomorphism of effective $t$-motives over $K$:
		\begin{equation}\label{E:efftmot}
		\Hom_{K[t]}(M_{\phi},\textbf{\textup{C}})\cong \big(\wedge_{K[t]}^{r-1}M_{\phi}\big)\otimes \Hom_{K[t]}(\det(M_{\phi}),\textbf{\textup{C}}).
		\end{equation}
		Moreover, the effective $t$-motive $\wedge_{K[t]}^{r-1}M_{\phi}$  is an abelian $t$-motive whose corresponding abelian $t$-module $\widetilde{\mathcal{E}}_0:=\wedge^{r-1}G_0:=(\mathbb{G}_{a/K}^{r-1},\widetilde{\varphi}_0)$ is defined by $\widetilde{\varphi}_0(t):=\theta \Id_{r-1}+E_1\tau+E_2\tau^2$ so that the matrices $E_1, E_2\in \Mat_{r-1}(K)$ are given as 
		\begin{equation}\label{E:matrices20}
		E_1:=(-1)^{r-1}\begin{bmatrix}
		-c_{r-1}&c_r& &  \\
		\vdots& &\ddots&  \\
		-c_{2}&0&\cdots &c_r\\
		-c_{1}&0&\cdots &0
		\end{bmatrix}\text{ and }E_2:=\begin{bmatrix}
		0&\cdots&0\\
		\vdots	& &\vdots\\
		c_r&\cdots &0
		\end{bmatrix}.
		\end{equation}
	\end{proposition}
	\begin{proof}  Since $M_{\phi}$ is a free $K[t]$-module of rank $r$, $\wedge_{K[t]}^{r-1}M_{\phi}$ is also free over $K[t]$ of rank $r$. Let us pick the $K[t]$-basis $\{g_1,\dots,g_r\}$ for $\wedge_{K[t]}^{r-1}M_{\phi}$ given by $g_1:=m_1\wedge \dots \wedge \tau^{r-2}m_1$,  $g_r:=\tau m_1\wedge \dots \wedge \tau^{r-1}m_1$ and
		\[
		g_i:=m_1\wedge \cdots \wedge\tau^{r-(i+1)} m_1\wedge\tau^{r-(i-1)} m_1\wedge \cdots \wedge \tau^{r-1}m_1
		\]
		for $2\leq i \leq r-1$. Since the left $\tau$-action  on $\wedge_{K[t]}^{r-1}M_{\phi}$ is diagonal, one can obtain
		\[
		\tau g_1=g_r
		\]
		and
		\[
		\tau g_i=(-1)^{r-2}c_r^{-1}(t-\theta)g_{i-1}+(-1)^{i-1}c_r^{-1}c_{r-(i-1)}g_r, \ \ 2\leq i \leq r.
		\]
		
		Recall that the left $K[t,\tau]$-module $\det(M_\phi)$ has rank $1$ as a $K[t]$-module with the basis $m'\in \det(M_\phi)$ where $\tau$ acts on $m'$ by 
		$
		\tau m'=(-1)^{r-1}c_r^{-1}(t-\theta)m'.
		$ 
		Therefore, the left $K[t,\tau]$-module $\Hom_{K[t]}(\det(M_{\phi}),\textbf{C})$ has a $K[t]$-basis $\{\tilde{m}\}$ such that 
		\[
		\tau \tilde{m}=(-1)^{r-1}c_r\tilde{m}
		\]
		and hence is isomorphic to the effective $t$-motive $\textbf{1}_{(-1)^{r-1}c_r}$ as left $K[t,\tau]$-modules. Note that $\{g_1\otimes \tilde{m},\dots,g_r\otimes \tilde{m}\}$ forms a $K[t]$-basis for $\wedge_{K[t]}^{r-1}M_{\phi}\otimes \Hom_{K[t]}(\det(M_{\phi}),\textbf{C})$ and furthermore, considering the diagonal $\tau$-action on  $\wedge_{K[t]}^{r-1}M_{\phi}\otimes \Hom_{K[t]}(\det(M_{\phi}),\textbf{C})$, we have 
		\begin{equation}\label{E:tau1}
		\tau (g_1\otimes \tilde{m})=(-1)^{r-1}c_r(g_r\otimes \tilde{m})
		\end{equation}
		and for $2\leq i \leq r$,
		\begin{equation}\label{E:tau2}
		\begin{split}
		\tau (g_i\otimes \tilde{m})&=(-1)^{r-2}(-1)^{r-1}(t-\theta)g_{i-1}\otimes \tilde{m} -(-1)^{r-1}(-1)^{i-2}c_{r-(i-1)}g_r\otimes \tilde{m}\\
		&=-(t-\theta)g_{i-1}\otimes \tilde{m}+(-1)^{r-i}c_{r-(i-1)}g_r\otimes \tilde{m}.
		\end{split}
		\end{equation}
		On the other hand, for the left $K[t,\tau]$-module $\Hom_{K[t]}(M_{\phi},\textbf{C})$, consider the  $K[t]$-basis $\{f_1,\dots,f_r\}$, where each $f_i$ is a $K[t]$-linear map defined by $f_i(m_k)=m$ if $i=k$ and $f_i(m_k)=0$ otherwise, where $m$ is the $K[t]$- basis for $\textbf{C}$. 
Then, we have that 
			\begin{equation}\label{E:tau3}
			\tau f_{i}=c_if_1+(t-\theta)f_{i+1}
			\end{equation}
			for $1\leq i \leq r-1$, and 
			\begin{equation}\label{E:tau4}
			\tau f_{r}=c_rf_1.
			\end{equation}
			Then, \eqref{E:tau3} and \eqref{E:tau4} now imply that $\Hom_{K[t]}(M_{\phi},\textbf{C})$ is an effective $t$-motive. Let 
			\[
			\iota:\big(\wedge_{K[t]}^{r-1}M_{\phi}\big)\otimes \Hom_{K[t]}(\det(M_{\phi}),\textbf{C})\to \Hom_{K[t]}(M_{\phi},\textbf{C})
			\]
			be the $K[t]$-module homomorphism defined as 
			$\iota(g_i\otimes \tilde{m})=(-1)^{r-i}f_{r-(i-1)}$. It is clear that $\iota$ is a $K[t]$-module isomorphism. Moreover, by using the left $\tau$-action \eqref{E:tau1}--\eqref{E:tau4} on the $K[t]$-bases $\{g_1\otimes \tilde{m},\dots,g_r\otimes \tilde{m}\}$ of $\big(\wedge_{K[t]}^{r-1}M_{\phi}\big)\otimes\Hom_{K[t]}(\det(M_{\phi}),\textbf{C})$ and $\{f_1,\dots,f_r\}$ of $\Hom_{K[t]}(M_{\phi},\textbf{C})$, we have for any $\alpha \in K[t]$ 
			\[
			\iota(\tau(\alpha g_1\otimes \tilde{m}))=\alpha^{(1)}\iota\big((-1)^{r-1}c_{r}g_r\otimes \tilde{m}\big)=\alpha^{(1)}(-1)^{r-1}c_{r}f_1=\tau(\iota(\alpha g_1\otimes \tilde{m})
			\]
			and for $2\leq i \leq r$
			\begin{align*}
			\iota(\tau(\alpha g_i\otimes \tilde{m}))&=\iota\big(\alpha^{(1)}(-(t-\theta)g_{i-1}\otimes \tilde{m}+(-1)^{r-i}c_{r-(i-1)}g_r\otimes \tilde{m})\big)\\
			&=\alpha^{(1)}(-1)^{r-i}\big((t-\theta)f_{r-(i-2)}+c_{r-(i-1)}f_1\big)\\
			&=\tau(\iota(\alpha g_i\otimes \tilde{m})
			\end{align*}
			which imply that $\iota$ is a left $K[t,\tau]$-module isomorphism. 
			
			 Using the $K[t]$-bases $\{f_1,\dots,f_r\}$ and  $\{m'\}$ of the left $K[t,\tau]$-modules $\Hom_{K[t]}(M_{\phi},\textbf{C})$ and $\textbf{1}_{(-1)^{r-1}c_r^{-1}}\cong(\Hom_{K[t]}(\det(M_{\phi}),\textbf{C}))^{\vee}$ respectively, by setting $f_i':=f_i\otimes m'$ for each $1\leq i \leq r$, one obtains a $K[t]$-basis $\{f_1',\dots,f_r'\}$ of $\wedge_{K[t]}^{r-1}M_{\phi}$  so that  
			\begin{equation}\label{E:extmap}
			\tau \begin{bmatrix} f_1'\\
			f_2'\\\vdots\\f_{r-1}'\\f_r'\end{bmatrix}=(-1)^{r-1}\begin{bmatrix}
			c_r^{-1}c_1 & c_r^{-1}(t-\theta) & 0 & \dots & 0\\
			c_r^{-1}c_2 & 0 & c_r^{-1}(t-\theta)&&\vdots\\
			\vdots&\vdots&\ddots&\ddots&0\\
			c_r^{-1}c_{r-1}&\vdots&\ddots&\ddots&c_r^{-1}(t-\theta)\\
			1 & 0&\dots&\dots&0
			\end{bmatrix}\begin{bmatrix} f_1'\\
			f_2'\\\vdots\\f_{r-1}'\\f_r'\end{bmatrix}.
			\end{equation}
			 Using \eqref{E:extmap}, we have that $\{f_r',\dots,f_2'\}$ is a left $K[\tau]$-basis for $\wedge_{K[t]}^{r-1}M_{\phi}$  and we obtain 
			\[
			t\cdot [f_r',\dots,f_2']^{\tr}=\widetilde{\varphi}_0(t)[f_r',\dots,f_2']^{\tr}.
			\]
			Thus, the latter assertion follows.
		\end{proof}

       Let $n$ be a positive integer. Using the $K[\tau]$-basis of $(\wedge_{K[t]}^{r-1}M_{\phi})\otimes \textbf{C}^{\otimes n}$ given by
        \[
        \{f_r', \dots, f_1'
		, (t-\theta)f_r', \dots,(t-\theta) f_1', \dots \dots, (t-\theta)^{n-1}f_r', \dots, (t-\theta)^{n-1}f_1', (t-\theta)^{n}f_r', \dots, (t-\theta)^nf_2'\},
        \]
        one can see that its corresponding abelian $t$-module $\widetilde{\mathcal{E}}_n:=(\wedge^{r-1}G_0)\otimes C^{\otimes n}:=(\mathbb{G}_{a/K}^{rn+r-1}, \widetilde{\varphi}_n)$  is defined by 
		\[
		\widetilde{\varphi}_n(t):=\theta \Id_{rn+r-1}+N'+E'\tau
		\]
		so that the matrices $N'\in \Mat_{rn+r-1}(\mathbb{F}_q)$ and $E'\in \Mat_{rn+r-1}(K)$ are given as 
		
		\begin{equation}\label{E:matrices2}
		N':=\begin{bmatrix}
		0&\cdots &0&\bovermat{$rn-1$}{1& 0&\cdots & 0} \\
		& \ddots& & \ddots& \ddots & &\vdots \\
		& &\ddots& &\ddots & \ddots&0\\
		& &  & 0&\cdots&0 & 1\\
		& &  & & 0&\cdots& 0\\
		& &  & & &\ddots& \vdots\\
		& &  & & & & 0\\
		\end{bmatrix}\begin{aligned}
		&\left.\begin{matrix}
		\\
		\\
		\\
		\\
		\end{matrix} \right\} %
		rn-1\\ %
		&\left.\begin{matrix}
		\\
		\\
		\\
		\end{matrix}\right\}%
		r
		\end{aligned}
		\end{equation}
		and
		\begin{equation}\label{E:matrices2A}
		E':=(-1)^{r-1}\begin{bmatrix}
		0&\cdots&\cdots & \cdots &\cdots & \cdots & 0\\
		\vdots& & & & & &\vdots\\
		0& & & & & & 0\\
		1& 0 & \cdots&\cdots&\cdots &\cdots &0\\
		-c_{r-1}&c_r&\ddots &  & & &\vdots\\
		\vdots& & \ddots& \ddots & & & \vdots\\
		-c_{1}&0&\cdots &c_r&0&\cdots& 0
		\end{bmatrix}\begin{aligned}
		&\left.\begin{matrix}
		\\
		\\
		\\
		\end{matrix} \right\} %
		rn-1\\ %
		&\left.\begin{matrix}
		\\
		\\
		\\
		\\
		\end{matrix}\right\}%
		r\\
		\end{aligned}.
		\end{equation}
		Here, the last $r$-rows of $N'$ contain only zeros which will be fundamental in proving our main results (see \S4).

\begin{remark} \label{R:dualmot} By \eqref{E:inthom} and Proposition \ref{P:prop1}, when $G_0=(\mathbb{G}_{a/K},\phi)$ is a Drinfeld $\bA$-module of rank $r\geq 2$, we have, as Taelman $t$-motives, 
	\begin{equation}\label{E:dual}
	\begin{split}
	M_{\phi}^{\vee}&=(\Hom_{K[t]}(M_{\phi},\textbf{C}),-1)\\
	&\cong ((\wedge_{K[t]}^{r-1}M_{\phi})\otimes \Hom_{K[t]}(\det(M_{\phi}),\textbf{C}),-1)\\
	&=(\wedge_{K[t]}^{r-1}M_{\phi})\otimes_{\mathcal{M}_K}\det(M_{\phi})^{\vee}.
	\end{split}
	\end{equation}
\end{remark}

		\section{An interpretation of Theorem \ref{T:mainThm} in terms of Taelman \texorpdfstring{$L$}{L}-values} 
		
		In this section, we introduce the $L$-function of a given abelian $t$-motive and relate its special values to Taelman $L$-values of explicit families of $t$-modules isomorphic to $\widetilde{\mathcal{E}}_n$ or $G_n$ described in \S2.4 and \S2.5 (see \cite[Sec. 2.3]{G94} and \cite[Sec. 2.8]{Taelman3} for $L$-functions attached to more general effective $t$-motives). Our purpose is to state Theorem \ref{T:3} and Theorem \ref{T:2} which are equivalent to the first and the second part of Theorem \ref{T:mainThm} respectively.
		
		\subsection{\texorpdfstring{$L$}{L}-function of abelian \texorpdfstring{$t$}{t}-motives} Consider an irreducible polynomial $\beta\in A_{+}$ and set $L:=A/\beta A$. For any abelian $t$-module $G$ of dimension $s$ defined over $K$, by clearing out the denominators of the entries of the coefficients if necessary, one can obtain an abelian $t$-module $\mathfrak{G}=(\mathbb{G}_{a/K}^{s},\varphi)$ defined over $A$ which is isomorphic to $G$ over $K$. Let us set $\overline{\mathfrak{G}}:=(\mathbb{G}_{a/L}^{s},\overline{\varphi})$ to be the $t$-module where $\overline{\varphi}(a)$ is the reduction of $\varphi(a)$ modulo $\beta$ for each $a\in\bA$. Let $M\in \mathfrak{M}_K$ be the abelian $t$-motive corresponding to $G$ which has rank $r$ over $K[t]$. We say that $M$ has \textit{good reduction at $\beta$} if there exists an abelian $t$-module $\mathfrak{G}$ defined over $A$ isomorphic to $G$ over $K$ so that the abelian $t$-motive $M_{\overline{\mathfrak{G}}}$ over $L$ corresponding to the abelian $t$-module $\overline{\mathfrak{G}}$  has rank $r$ over $L[t]$. 
		
		Let $K^{\text{sep}}$ be the separable closure of $K$ in $\CC_{\infty}$ and define $M_{K^{\text{sep}}}:=M\otimes_{K} K^{\text{sep}}$. Let $v$ be a monic irreducible element in $\textbf{A}$. Consider  $\textbf{A}_v:=\varprojlim_{i}\textbf{A}/v^i\textbf{A}$ and set $\textbf{K}_v:=\textbf{A}_v\otimes_{\textbf{A}}\mathbb{F}_q(t)$. We define
		\[
		\mathcal{H}_{v}(M):=\varprojlim_{i}(M_{K^{\text{sep}}}/v^iM_{K^{\text{sep}}})^{\tau}
		\]
		where, for each $i$,  $(M_{K^{\text{sep}}}/v^iM_{K^{\text{sep}}})^{\tau}$ denotes the elements of  $M_{K^{\text{sep}}}/v^iM_{K^{\text{sep}}}$ fixed by the action of $\tau$. Note that $\mathcal{H}_{v}(M)$ is a free $\bA_v$-module of rank $r$ with a continuous action of $\Gal(K^{\text{sep}}/K)$ (see \cite[Sec.5.6]{Goss} for more details). 
		
		Let $\beta'$ be a prime of $K^{\text{sep}}$ lying above $\beta$ and  $I_{\beta'}\subset \Gal(K^{\text{sep}}/K)$ be the inertia group at $\beta'$. We set $\mathcal{H}_{v}(M)^{I_{\beta'}}$ to be the $\bA_v$-module consisting of elements of $\mathcal{H}_{v}(M)$ invariant under the action of $I_{\beta'}$. Let $\overline{L}$ be a fixed algebraic closure of $L$ and let $\Frob_\beta^{-1}\in \Gal(\overline{L}/L)$ be the geometric Frobenius at $\beta$. We define \textit{the local factor of $M$ at $\beta$} to be the following polynomial:
		\[
		P_{\beta}^{M}(X):=\det_{\bA_v}(1-X\Frob_\beta^{-1}\ \ |\mathcal{H}_{v}(M)^{I_{\beta'}})\in \bA_v[X].
		\]
		Note that the definition of $P_{\beta}^{M}(X)$ is indeed independent of the choice of ${\beta'}$ \cite[Sec. 8.6]{Goss}. 
        
        The following result is due to Gardeyn.
		\begin{theorem}\cite[Thm. 1.1, Thm. 7.3]{G01} \label{T:Gar}For any irreducible element $v\in \bA$ so that $v(\theta)$ is not equal to $\beta$,  $I_{\beta}$ acts trivially on $\mathcal{H}_{v}(M)$ if and only if $M$ has good reduction at $\beta$. Moreover, for such $v$ and $\beta$, the polynomial $P_{\beta}^{M}(X)$ has coefficients in $\bA$ and is independent of the choice of $v$. 
		\end{theorem}
		There exists a finite set $S$ of primes in $A$ where $M$ does not have good reduction. Consider the following infinite product
		\[
		L_S(M,n):=\prod_{\beta\not\in S}P^{M}_\beta(\beta_{|\theta=t}^{-n})^{-1}.
		\]
		Now for any natural number $n$, we construct the $L$-function $L(M,n)$ of $M$ by 
		\begin{equation}\label{E:lfunctiondef}
L(M,n):=L_S(M,n)\prod_{\beta\in S}P^{M}_\beta(\beta_{|\theta=t}^{-n})^{-1}.
		\end{equation}
				Let $\beta\not \in S$. We write $P_{\beta}^{M}(X)=(1-\alpha_1X)\cdots (1-\alpha_rX)$ for some $\alpha_1,\dots,\alpha_r\in \overline{\bK_v}$ where $\overline{\bK_v}$ is a fixed algebraic closure of $\bK_v$. We further define $P_{\beta}^{M^{\vee}}(X):=(1-\alpha_1^{-1}X)\cdots (1-\alpha_r^{-1}X)\in\bK_v[X]$ and consider 
		\[
L_S(M^{\vee},n):=\prod_{\beta\not\in S}P^{M^{\vee}}_\beta(\beta_{|\theta=t}^{-n})^{-1}.
		\]
		\begin{remark}\label{R:algebraic}
        \begin{itemize}
        \item[(1)] In what follows, let the notation be as above. For later use, we also consider the polynomial $Q_{\beta}^M(X)$ defined by the relation $P_{\beta}^M(X)=X^{r}Q_{\beta}^M(X^{-1})$. Indeed, we have 
			\[
			Q_{\beta}^M(X)=(X-\alpha_1)\cdots (X-\alpha_r).
			\]
			We note that
			\begin{equation}\label{E:dualpoly}
			\frac{Q_{\beta}^M(1)}{Q_{\beta}^M(0)}=\frac{(1-\alpha_1)\cdots(1-\alpha_r)}{(-1)^r\alpha_1\cdots \alpha_r}=(1-\alpha_1^{-1})\cdots (1-\alpha_r^{-1})=P_{\beta}^{M^{\vee}}(1).
			\end{equation}
            \item[(2)] It is important to note that due to the construction, we need to distinguish the variable $t$ and $\theta$ which causes our $L$-functions to converge in $\mathbb{F}_q((1/t))$. However, for the rest of the paper, we instead consider these values in $K_{\infty}$ by simply replacing $t$ with $\theta$.
            \item[(3)] When $M$ is the abelian $t$-motive $M_{\phi}$ corresponding to $G_0=(\mathbb{G}_{a/K},\phi)$, by \cite[Cor. 8.4]{G01}, we know that 
		$\prod_{\beta\in S}P^{M_{\phi}}_\beta(\beta_{|\theta=t}^{-n})^{-1}$ is an element in $\mathbb{F}_q(t)$. Hence, one can neglect $\prod_{\beta\in S}P^{M_{\phi}}_\beta(\beta_{|\theta=t}^{-n})^{-1}$ in \eqref{E:lfunctiondef} to show the transcendence of special values over $K$ as this term, after evaluating it at $t=\theta$, lies in $K$. 
        \end{itemize}
		\end{remark}

		As an example, the value of the $L$-function $L(\textbf{C},s)=L_{\emptyset}(\textbf{C},s)$ at $s\in \ZZ_{\geq 2}$ corresponding to the abelian $t$-motive $\textbf{C}$ is given by 
		\[
		L(\textbf{C},s)=\sum_{a\in A_{+}}\frac{1}{a^{s-1}}\in K_{\infty}.
		\]
		
		Consider a Taelman $t$-motive $\mathbb{M}=(M,i)$ so that $M$ is an abelian $t$-motive over $K$ and has good reduction at primes of $A$ outside of $S$. Observe that 
		\begin{equation}\label{E:Lfunc}
		L_S(\mathbb{M}\otimes_{\mathcal{M}_K} \textbf{C}, n+1)=L_S(\mathbb{M},n)
		\end{equation}
		for any $\mathbb{M}\in \mathcal{M}$ and sufficiently large $n$ which simply follows from the tensor compatibility of the functor $\mathcal{H}_v$ (see \cite[Sec. 2.3.5]{HartlJuschka16}) and the fact that the polynomial $P_\beta(X)$ for $\textbf{C}$ is given by $1-X\beta_{|\theta=t}$.
		
		We continue to assume that $v$ is a monic irreducible polynomial so that $v(\theta)\neq\beta$. We finish this subsection by defining \textit{the $v$-adic Tate module} which will be fundamental to analyze local factors of $L$-functions. Let $\mathfrak{G}=(\mathbb{G}_{a/K}^{s},\varphi)$ be an abelian $t$-module  defined over $A$ whose corresponding abelian $t$-motive has rank $r$ over $K[t]$.  Assume that $\overline{\mathfrak{G}}$ is an abelian $t$-module over $L$ whose corresponding abelian $t$-motive has rank $r$ over $L[t]$.

		Let $L^{\text{sep}}$ be a fixed separable closure of $L$. For any $n\geq 1$, we define \textit{the set of $v^n$-torsion points of $\overline{\mathfrak{G}}$} by
		\[
		\overline{\mathfrak{G}}[v^n]:=\{f\in \overline{\mathfrak{G}}(L^{\sep}) \ \ | \ \ v^n\cdot f=0\}.
		\]
		By \cite[Cor. 5.6.4]{Goss}, we know that $\overline{\mathfrak{G}}[v^n]$ is a finite $\bA$-module isomorphic to $\Big(\bA/v^n\bA\Big)^{\oplus r}$. We define \textit{the $v$-adic Tate module of $\mathfrak{G}$} by 
		\[
		T_v(\mathfrak{G}):=\varprojlim_{n}\overline{\mathfrak{G}}[v^n].
		\]
		By \cite[Thm. 5.6.8]{Goss}, it is a free $\bA_v$-module of rank $r$ which also has a continuous action of $\Gal(L^{\text{sep}}/L)$.
		
		\subsection{Explicitly constructed abelian \texorpdfstring{$t$}{t}-motives}
		Recall the Drinfeld $\bA$-module $G_0=(\mathbb{G}_{a/K},\phi)$ of rank $r\geq 2$ defined as in \eqref{E:drinfeldgen}. Recall also its abelian $t$-motive $M_{\phi}\in \mathfrak{M}_K$ from Example \ref{Ex:1}(i).
		
		Let $n$ be a non-negative integer. In what follows, we investigate families of abelian $t$-motives and their corresponding abelian $t$-modules. Moreover, we analyze their associated $L$-functions by assuming that they are well-defined elements in $K_{\infty}$ which will be clear to the reader in \S3.4 after their comparison with certain Taelman $L$-values.
		
		\subsubsection{\textbf{\textup{Construction of $\textbf{E}_n$ and $\textbf{E}_{n,A}$}}} In this subsection, we let $\mathbb{G}_0=(\mathbb{G}_{a/K},\phi)$ be a Drinfeld $\bA$-module of rank $r\geq 2$. Consider the effective $t$-motive $\textbf{E}_n$ given by 
		\[
		\textbf{E}_n:=(\wedge_{K[t]}^{r-1}M_{\phi})\otimes_{\mathcal{M}_K} \textbf{C}^{\otimes (n+1)} \otimes_{\mathcal{M}_K} \det(M_{\phi})^{\vee}\cong \big(\wedge_{K[t]}^{r-1}M_{\phi}\big)\otimes \Hom_{K[t]}(\det(M_{\phi}),\textbf{\textup{C}})\otimes \textbf{C}^{\otimes n} \in \mathcal{M}_K
		\]
		which is indeed isomorphic to an effective $t$-motive over $K$ and the second isomorphism follows from \eqref{E:dualidentity}. Note that, using the correspondence described in \S 2.2, one can describe the $t$-module attached to $\textbf{E}_n$. More precisely, we first
         recall the abelian $t$-module $\widetilde{\mathcal{E}}_0=(\mathbb{G}^{r-1}_{a/K},\widetilde{\varphi}_0)$ from the statement of Proposition \ref{P:prop1} as well as the matrices $E_1$ and $E_2$.  Set $\gamma:=((-1)^{r-1}c_r^{-1})^{1/(q-1)}\in \oK$. When $n=0$,  the abelian $t$-module $\mathcal{E}_0:=(\mathbb{G}_{a/K}^{r-1},\varphi_0)$ attached to $\textbf{E}_0$ is given by $\gamma^{-1}\widetilde{\varphi}_0\gamma=:\varphi_0:\bA\to \Mat_{r-1}(K)[\tau]$  which can be explicitly stated as 
		\begin{equation*}\label{E:motiveExt}
		\varphi_0(t):=\theta \Id_{r-1}+(-1)^{r-1}c_r^{-1}E_1\tau+c_r^{-q-1}E_2\tau^2\in \Mat_{r-1}(K)[\tau].
		\end{equation*}
       On the other hand, recall the abelian $t$-module $\widetilde{\mathcal{E}}_n=(\mathbb{G}^{rn+r-1}_{a/K},\widetilde{\varphi}_n)$ where $n\geq 1$ and the matrices $N'$ and $E'$ from \S2.5. Then, the abelian $t$-module $\mathcal{E}_n:=(\mathbb{G}_{a/K}^{rn+r-1},\varphi_n)$ attached to $\textbf{E}_n$ is given by $\gamma^{-1}\widetilde{\varphi}_n\gamma=:\varphi_n:\bA\to \Mat_{rn+r-1}(K)[\tau]$ where
		\begin{equation*}\label{E:motiveExt2}
		\varphi_n(t):=\theta \Id_{rn+r-1}+N'+(-1)^{r-1}c_r^{-1}E'\tau\in \Mat_{rn+r-1}(K)[\tau].
		\end{equation*}
		
        Now for $1\leq i \leq r$, we set $\mathfrak{a}_{i}:=1$ if $c_i$ is zero and if $c_i\neq 0$, we write $c_i=\frac{\mathfrak{b}_i}{\mathfrak{a}_i}$ where $\mathfrak{b}_i$ and $\mathfrak{a}_i$ are relatively prime polynomials in $A$. We let $\mathfrak{a}:=\mathfrak{a}_1\cdots \mathfrak{a}_{r-1}\mathfrak{b}_r\in A\setminus \{0\}$. Then, for each $n\geq 0$, we see that the abelian $t$-module $\mathcal{E}_{n,A}:=(\mathbb{G}_{a/K}^{rn+r-1},\varphi_{n,A})$, given by  $\varphi_{n,A}(t):=\mathfrak{a}^{-1}\varphi_n(t)\mathfrak{a}$, is defined over $A$. Let $\textbf{E}_{n,A}\in \mathfrak{M}_K$ be its corresponding abelian $t$-motive. It is clear that $\textbf{E}_{n,A}  \cong \textbf{E}_{n}$ as $K[t,\tau]$-modules.

		Since taking the dual of an effective $t$-motive is reflexive (see  \cite[Sec. 2.3.5]{Taelman2}), by \eqref{E:dual}, we have an isomorphism of Taelman $t$-motives
		\begin{equation}\label{E:Lfunc6}
		(\wedge_{K[t]}^{r-1}M_{\phi})^{\vee}\cong M_{\phi}\otimes_{\mathcal{M}_K}\det(M_{\phi})^{\vee}.
		\end{equation}
		
		On the other hand, using \eqref{E:dualidentity}and \eqref{E:efftmot}, one can see that  
        \[
        M_{\phi}^{\vee}\otimes \textbf{C}^{\otimes (n+1)}\cong\Hom_{K[t]}(M_{\phi},\textbf{C})\otimes \textbf{C}^{\otimes n}\cong \big(\wedge_{K[t]}^{r-1}M_{\phi}\big)\otimes \Hom_{K[t]}(\det(M_{\phi}),\textbf{\textup{C}})\otimes \textbf{C}^{\otimes n} \cong \textbf{E}_n
        \]
        Dualizing  above, we obtain
		\begin{equation}\label{E:dualize}
(\textbf{E}_n)^{\vee}=\Hom(\textbf{E}_n,\textbf{1})\cong M_{\phi}\otimes_{\mathcal{M}_K}(\textbf{C}^{\otimes (n+1)})^{\vee}.
		\end{equation}
		
		Assume that $S$ is the subset of primes of $A$ where $M_{\phi}$ does not have good reduction. Thus, by \eqref{E:dualize}, we obtain
		\begin{equation}\label{E:Lfunc44}
		L_S((\textbf{E}_{n,A})^{\vee},0)=L_S((\textbf{E}_n)^{\vee},0)=L_S(M_{\phi}\otimes_{\mathcal{M}_K}(\textbf{C}^{\otimes (n+1)})^{\vee},0)=L_S(M_{\phi},n+1).
		\end{equation}

		\subsubsection{\textbf{\textup{Construction of $\textbf{G}_{n,A}$}}} In this case, we let $G_0=(\mathbb{G}_{a/K},\phi)$ be a Drinfeld $\bA$-module which is isomorphic, over $K$, to $\psi$ given as in \eqref{E:drinfeld} and let $M_{\psi}\in \mathfrak{M}_K$ be its corresponding abelian $t$-motive. Consider the Taelman $t$-motive $\textbf{G}_{n,A}$ given by 
		\begin{equation*}\label{E:motive}
	\textbf{G}_{n,A}:=M_{\psi}\otimes_{\mathcal{M}_K} \textbf{C}^{\otimes (n+1)} \otimes_{\mathcal{M}_K} \det(M_{\psi})^{\vee}\in \mathcal{M}_K.
		\end{equation*}
		Indeed, one can check that $\textbf{G}_{n,A}$ is an abelian $t$-motive whose corresponding abelian $t$-module $G_{n,A}:=(\mathbb{G}_{a/K}^{rn+1},\psi_{n,A})$  is given as 
		\begin{multline}\label{E:motive2}
		\psi_{n,A}(t):=((-1)^{r-1}a_r^{-1})^{-1/(q-1)}\psi_n(t)((-1)^{r-1}a_r^{-1})^{1/(q-1)}\\=\theta \Id_{rn+1}+N+(-1)^{r-1}a_r^{-1}E\tau\in \Mat_{rn+1}(A)[\tau]
		\end{multline}
		where the matrices $N$ and $E$ are as in \eqref{E:matrices} and \eqref{E:matricesA} respectively. Using \eqref{E:dual2} and \eqref{E:dual}, one can immediately see that 
		\[
		\textbf{G}_{n,A}^{\vee}\cong(\wedge_{K[t]}^{r-1}M_{\psi})\otimes_{\mathcal{M}_K}(\textbf{C}^{\otimes (n+1)})^{\vee}.
		\]

		On the other hand, since $a_r\in \mathbb{F}_q^{\times}$, it can be seen by \eqref{E:matrices20} and \eqref{E:extmap} that $\wedge_{K[t]}^{r-1}M_{\psi}$ has good reduction at any prime of $A$. Hence applying \eqref{E:Lfunc} repeatedly, we obtain
		\begin{equation}\label{E:Lfunc4}
		L(\textbf{G}_{n,A}^{\vee},0)=L(\wedge_{K[t]}^{r-1}M_{\psi}\otimes_{\mathcal{M}_K}(\textbf{C}^{\otimes (n+1)})^{\vee},0)=L(\wedge_{K[t]}^{r-1}M_{\psi},n+1)=L(\wedge_{K[t]}^{r-1}M_{\phi},n+1)
		\end{equation}
		where the last equality follows from the fact that $M_{\phi}\cong M_{\psi}$ by assumption.
		
		We finish this subsection with the following proposition which is a consequence of \cite[Rem. 7.3.5]{Thakur}.
		\begin{proposition}\label{P:pure} Assume that $\mathbf{E}_{n,A}$ and $\mathbf{G}_{n,A}$ have good reduction at a prime $\beta\in A_{+}$. Then, the abelian $t$-modules  $\overline{\mathcal{E}_{n,A}}$ and $\overline{G_{n,A}}$ defined over $L=A/\beta A$ are pure.
		\end{proposition}
		
		\subsection{Local factors at places with good reduction}  
		
		For any finite $\bA$-module $\mathrm{N}$ so that $\mathrm{N}\cong \oplus_{i=1}^m \bA/f_i\bA$ for some elements $f_1,\dots,f_m\in \bA$, we define $|\mathrm{N}|_{\bA}$ to be the monic generator of the principal ideal $(f_1\cdots f_m)\subseteq\bA$.
		
		Recall that $\beta$ is an irreducible polynomial in $ A_{+}$ and let $v\in \bA$ so that $v(\theta)\neq \beta$. As we proceed in the previous subsection, we divide our analysis into two cases which correspond to the calculation for the local factors of the $L$-functions corresponding to the dual of the abelian $t$-motives defined in \S3.2.
		
		\subsubsection{\textbf{\textup{Local factors of the $L$-function of $\textbf{E}_{n,A}^{\vee}$} }} We first analyze the local factors of the $L$-function of $\textbf{E}_{n,A}$ defined as in \S3.2.1. Assume that $\textbf{E}_{n,A}$ has good reduction at $\beta$ and recall that $\overline{\textbf{E}_{n,A}}$ is the abelian $t$-motive over $L$ corresponding to $\overline{\mathcal{E}_{n,A}}$. Let $Q_\beta^{\textbf{E}_{n,A}}(X)\in \bA_v[X]$ be the characteristic polynomial of the $q^{\deg_{\theta}(\beta)}$-th power map $\tau^{\deg_{\theta}(\beta)}$ on $T_v(\overline{\mathcal{E}_{n,A}})$.  We have
		\begin{align*}
		P_{\beta}^{\textbf{E}_{n,A}}(X)=\det(1-X\tau^{\deg_{\theta}(\beta)} | \ \ \overline{\textbf{E}_{n,A}})=\det(1-X\tau^{\deg_{\theta}(\beta)}| \ \ T_v(\overline{\textbf{E}_{n,A}}))=X^rQ_\beta^{\textbf{E}_{n,A}}(X^{-1})
		\end{align*}
		where the first equality follows from \cite[Prop. 7]{Taelman3} (see also \cite[Sec. 7]{G01}) and the second equality from \cite[Sec. 6]{TagWan} (see also \cite[Prop. 5.6.9]{Goss}).
		
		Note that, in this case, by \eqref{E:efftmot}, we have 
		\[
		P_{\beta}^{\textbf{E}_{n,A}}(X)=d(1-\alpha_1^{-1}\beta^{n+1} X)\cdots (1-\alpha_r^{-1}\beta^{n+1} X)
		\]
		for some $d\in\mathbb{F}_q^{\times}$. Thus, by \cite[Thm. 5.1]{Gek91} (see also \cite[Thm. 3.1, Cor. 3.2]{ChangEl-GuindyPapanikolas}), we obtain
		\begin{equation}\label{E:polynQ}
		Q_\beta^{\textbf{E}_{n,A}}(X)=b_0+b_1X+\dots+b_{r-1}X^{r-1}+X^r\in \bA[X]
		\end{equation}
		where $\deg_{t}(b_i)<\deg_\theta(\beta^{rn+r-1})$ for $1\leq i \leq r-1$ and $b_0=c\beta_{|\theta=t}^{rn+r-1}$ for some $c\in \mathbb{F}_q^{\times}$. 
		
		For any $a\in \bA$, let $(a)_v$ be the $v$-primary part of the principal ideal $(a)\subseteq \bA$ generated by $a$. By Proposition \ref{P:pure} and \cite[Thm. 5.6.10]{Goss} (see also \cite[Thm. 7.8]{BH09}), we know that the eigenvalues of $\Frob_\beta$ has norm larger than one. This implies that the map $1-\Frob_\beta$ is injective on $T_v(\overline{\mathcal{E}_{n,A}})$. Therefore by \cite[Prop. 2.15]{Deb} and the fact that the set of $v^m$-torsion points of $\overline{\mathcal{E}_{n,A}}$ is finite for each $m\in \ZZ_{\geq 1}$ (\cite[Cor. 5.6.4]{Goss}), we have 
		\[
		(P_{\beta}^{\textbf{E}_{n,A}}(1))_v=(Q_\beta^{\textbf{E}_{n,A}}(1))_v=(|\overline{\mathcal{E}_{n,A}}(A/\beta A)|_{\bA})_v.
		\] 
		On the other hand, by \eqref{E:polynQ}, we have  $\deg_t(Q_\beta^{\textbf{E}_{n,A}}(1))=\deg_t(|\overline{\mathcal{E}_{n,A}}(A/\beta A)|_{\bA})=\deg_t(\beta)(rn+r-1)$. Hence we obtain 
		\[
		c^{-1}Q_\beta^{\textbf{E}_{n,A}}(1)=|\overline{\mathcal{E}_{n,A}}(A/\beta A)|_{\bA}.
		\]
		Since $\beta^{rn+r-1}_{|\theta=t}$ annihilates the $\bA$-module $\Lie(\overline{\mathcal{E}_{n,A}})$, one can also see that 
		\[
		(Q_\beta^{\textbf{E}_{n,A}}(0))=(\beta_{|\theta=t}^{rn+r-1})=(|\Lie(\overline{\mathcal{E}_{n,A}})(A/\beta A)|_{\bA}).
		\]
		It implies that $c^{-1}Q_\beta^{\textbf{E}_{n}}(0)=|\Lie(\overline{\mathcal{E}_n})(A/\beta A)|_{\bA}$. Hence, we finally obtain 
		\begin{equation}\label{E:locFacfit}
		\frac{|\overline{\mathcal{E}_{n,A}}(A/\beta A)|_{\bA}}{|\Lie(\overline{\mathcal{E}_{n,A}})(A/\beta A)|_{\bA}}=\frac{Q_\beta^{\textbf{E}_{n,A}}(1)}{Q_\beta^{\textbf{E}_{n,A}}(0)}=P_{\beta}^{\textbf{E}_{n,A}^{\vee}}(1)=P_{\beta}^{(\textbf{E}_{n})^{\vee}}(1)
		\end{equation}
		where the second equality follows from \eqref{E:dualpoly} and the last equality from the fact that $\textbf{E}_{n,A}\cong \textbf{E}_{n}$ over $K$.

		\subsubsection{\textbf{\textup{Local factors of the $L$-function of $\textbf{G}_{n,A}^{\vee}$}}}
		Now we analyze $Q_\beta^{\textbf{G}_{n,A}}(X)\in \bA_v[X]$, the characteristic polynomial of the $q^{\deg_{\theta}(\beta)}$-th power map on $T_v(\overline{G_{n,A}})$ using the idea as above.
		Note that, in this case, we have 
		\[
		P_{\beta}^{\textbf{G}_{n,A}}(X)=\tilde{d}(1-\alpha_1\beta^n X)\cdots (1-\alpha_r\beta^n X)
		\]
		for some $\tilde{d}\in \mathbb{F}_q^{\times}$. Thus, again by \cite[Thm. 5.11]{Gek91}, we obtain
		\[
		Q_\beta^{\textbf{G}_{n,A}}(X)=\tilde{b}_0+\tilde{b}_1X+\dots+\tilde{b}_{r-1}X^{r-1}+X^r\in \bA[X]
		\]
		where $\deg_{t}(\tilde{b}_i)<\deg_\theta(\beta^{rn+1})$ for $1\leq i \leq r-1$ and $\tilde{b}_0=\tilde{c}\beta_{|\theta=t}^{rn+1}$ for some $\tilde{c}\in \mathbb{F}_q^{\times}$. Moreover, using the same idea in \S3.3.1, we also obtain 
		\[
		\tilde{c}^{-1}Q_\beta^{\textbf{G}_{n,A}}(1)=|\overline{G_{n,A}}(A/\beta A)|_{\bA}
		\]
		and 
		\[
		\tilde{c}^{-1}Q_\beta^{\textbf{G}_{n,A}}(0)=|\Lie(\overline{G_{n,A}})(A/\beta A)|_{\bA}.
		\]
		Hence, by again \eqref{E:dualpoly}, we see that 
		\begin{equation}\label{E:locFacfit2}
		\frac{|\overline{G_{n,A}}(A/\beta A)|_{\bA}}{|\Lie(\overline{G_{n,A}})(A/\beta A)|_{\bA}}=\frac{Q_\beta^{\textbf{G}_{n,A}}(1)}{Q_\beta^{\textbf{G}_{n,A}}(0)}=P_{\beta}^{\textbf{G}_{n,A}^{\vee}}(1).
		\end{equation}

		\subsection{Taelman \texorpdfstring{$L$}{L}-values}
		We start by introducing Taelman $L$-values which were firstly defined by Taelman \cite{Taelman} for Drinfeld $\bA$-modules and later by Fang \cite{Fang} for abelian $t$-modules.
		
		Let $G=(\mathbb{G}_{a/K}^s,\varphi)$ be an abelian $t$-module of dimension $s$ defined over $A$. For any finite $\bA$-module $\mathrm{N}$, we now set $|\mathrm{N}|_{A}:=(|\mathrm{N}|_{\bA})_{|t=\theta}$. Let $\beta \in A_{+}$ be a prime element. \textit{The Taelman $L$-value $L(G/A)$} is defined by the infinite product
		\[
		L(G/A):=\prod_{\substack{\beta\in A_{+}\\\beta \text{ prime}}}\frac{|\Lie(\overline{G})(A/\beta A)|_{A}}{|\overline{G}(A/\beta A)|_{A}}\in 1+\frac{1}{\theta}\mathbb{F}_q\Big[\Big[\frac{1}{\theta}\Big]\Big].
		\]
		
		Now we let $G_0=(\mathbb{G}_{a/K},\phi)$ be the Drinfeld $\bA$-module of rank $r\geq 2$ as in \eqref{E:drinfeldgen}. Consider the Goss $L$-function $L(M_{\phi},n)$ defined as above and the finite set $S$ of places of $\bA$ where $M_{\phi}$ does not have good reduction. Thus, using Remark \ref{R:algebraic}(iii), we obtain 
		\begin{equation}\label{E:Gard}
		L(M_{\phi},n+1)=\alpha L_S(M_{\phi},n+1)
		\end{equation}
		for some $\alpha \in K^{\times}$ whenever both sides converge in $K_{\infty}$. Hence combining \eqref{E:Lfunc44}, \eqref{E:locFacfit}, and \eqref{E:Gard}, we obtain 
		\begin{equation}\label{E:comp1}
		L(M_{\phi},n+1)= \alpha L_S(M_{\phi},n+1)=\alpha L_S((\textbf{E}_{n,A})^{\vee},0)=\alpha' L(\mathcal{E}_{n,A}/A)
		\end{equation}
		for some $\alpha'\in K^{\times}$.

		We now obtain the following equivalent statement of the first part of Theorem \ref{T:mainThm}.
		\begin{theorem}\label{T:3}  For any non-negative integer $n$, let $\mathcal{E}_{n,A}$ be the abelian $t$-module defined as in \S3.2.1 and $L(\mathcal{E}_{n,A}/A)$ be its Taelman $L$-value. Then, $L(\mathcal{E}_{n,A}/A)$ is transcendental over $\oK$. 
		\end{theorem}
		
		Observe that \eqref{E:Lfunc4} and \eqref{E:locFacfit2} imply 
		\begin{equation}\label{E:comp2}
		L(\wedge^{r-1}_{K[t]}M_{\phi},n+1)=L(\textbf{G}_{n,A}^{\vee},0)=L(G_{n,A}/A).
		\end{equation}
        Indeed, in what follows, we can prove a more general statement which implies the second part of Theorem \ref{T:mainThm}.
		\begin{theorem}\label{T:2} Let $\phi$ be a Drinfeld $\bA$-module defined over $A$ and $G_n$ be the $t$-module defined as in \S2.4 constructed from $\phi$ and $C^{\otimes n}$. Let $L(G_{n}/A)$ be its Taelman $L$-value. Then, for any non-negative integer $n$, $L(G_{n,A}/A)$ is transcendental over $\oK$.
		\end{theorem}

        \begin{remark} Let $\psi$ be a Drinfeld $\bA$-module as in \eqref{E:drinfeld} and consider the Drinfeld $\bA$-module $\psi'$ defined over $A$ given by \[
    \psi'_t:=((-1)^{r-1}a_r^{-1})^{1/(q-1)}\psi_t((-1)^{r-1}a_r^{-1})^{-1/(q-1)}.
    \]
    Observe that $G_{n,A}$ is the abelian $t$-module constructed from $\psi'$ and $C^{\otimes n}$ as in \S2.4 and hence \eqref{E:comp2} and Theorem \ref{T:2} imply the second part of Theorem \ref{T:mainThm}.
        \end{remark}

        \begin{remark}\label{R:Carlitz}
        Before we finish this section, we briefly explain how to obtain Theorem \ref{T:mainThm} for the case of Drinfeld $\bA$-module $G_0=C_{\mathfrak{b}}$ of rank $1$ defined as in Example \ref{Ex:0} for $\mathfrak{b}\in K^{\times}$. Note that there exists $\widetilde{\mathfrak{b}}\in K^{\times}$ such that $C^{\otimes n}_{\mathfrak{b}}$ is isomorphic to $C^{\otimes n}_{\widetilde{\mathfrak{b}}^{-1}}$ over $A$.  Thus, by Example \ref{Ex:1}(ii), $\textbf{C}^{\otimes n}_{\mathfrak{b}^{-1}}\cong \textbf{C}^{\otimes n}_{\widetilde{\mathfrak{b}}}$ as $K[t,\tau]$-modules. We set $\mathfrak{c}:=\widetilde{\mathfrak{b}}^{-1}$. Using Example \ref{Ex:3}(i) and applying the discussion on local factors of the $L$-function of $\textbf{G}_{n,A}^{\vee}$ in \S3.3.2 to Drinfeld $\bA$-modules of rank $1$ defined over $A$, we obtain
        \[
L(\textbf{C}_{\mathfrak{b}},n+1)=\alpha^{''}L(C_{\mathfrak{c}}^{\otimes n}/A)
        \]
for some $\alpha^{''}\in K^{\times}$. On the other hand, by \cite[Thm.~4.4]{ADTR}, there exists a non-zero element $\mathfrak{g}\in \Lie(C_{\mathfrak{c}}^{\otimes n})(K_{\infty})$ so that $\Exp_{C_{\mathfrak{c}}^{\otimes n}}(\mathfrak{g})\in C_{\mathfrak{c}}^{\otimes n}(A)$. Moreover $L(C_{\mathfrak{c}}^{\otimes n}/A)=a'\mathfrak{g}_r$ where $\mathfrak{g}_r$ is the last coordinate of $\mathfrak{g}$ and $a'\in K$. Note that, as elements in $\Mat_{n}(\mathbb{C}_{\infty})[[\tau]]$, we have $\Exp_{C_{\mathfrak{c}}^{\otimes n}}=\mathfrak{c}^{-1/q-1}\Exp_{C^{\otimes n}}\mathfrak{c}^{1/q-1}$ for some fixed $(q-1)$-st root $\mathfrak{c}^{1/q-1}$ of $\mathfrak{c}$. Thus, by \cite[Thm. 2.3]{Yu91}, $\mathfrak{c}^{1/q-1}\mathfrak{g}_r\not \in \overline{K}$, implying that $L(C_{\mathfrak{c}}^{\otimes n}/A)\not \in \overline{K}$. 

We establish our main result for Drinfeld $\bA$-modules of rank $r\geq 2$ by proving Theorem \ref{T:3} and Theorem \ref{T:2} which will be done in the next section.
		\end{remark}

		\section{The Proof of Theorem \ref{T:3} and Theorem \ref{T:2}}
		In this section, our aim is to prove Theorem \ref{T:3} and Theorem \ref{T:2} which are equivalent to Theorem \ref{T:mainThm}(i) and Theorem \ref{T:mainThm}(ii) respectively. Since the methods we use to obtain the results are similar, we only explain the details of the proof of Theorem \ref{T:3} and explicitly state how to obtain Theorem \ref{T:2} from the ideas of the proof of Theorem \ref{T:3}.
		
		Let $G=(\mathbb{G}_{a/K}^d,\varphi)$ be an abelian $t$-module so that $\varphi(t)\in \Mat_d(A)[\tau]$. We define the unit module $U(G/A)$ of $G$ by 
		\[
		U(G/A):=\{x\in \Lie(G)(K_{\infty}) \ \ | \Exp_{G}(x)\in G(A)     \}.
		\]
		Using \eqref{E:FuncEqtn}, one can obtain that $U(G/A)$ is an $\bA$-submodule of $\Lie(G)(K_{\infty})$. Moreover, by \cite[Thm. 1.10]{Fang}, $U(G/A)$ is indeed a free $\bA$-module of rank $d$.
		
		We consider  the map $\rd_{\varphi}:\mathbb{F}_q((1/t))\to K_{\infty}$ defined by 
		\[
		\rd_{\varphi}\Big(\sum_{i\geq i_0}c_it^{-i}\Big)=\sum_{i\geq i_0}c_i\rd_{\varphi}(t)^{-i}, \ \ c_i\in \mathbb{F}_q
		\]
		and we further equip $\Lie(G)(K_{\infty})$ with the $\mathbb{F}_q((1/t))$-vector space structure given by 
		\[
		f\cdot x=\rd_{\varphi}(f)x, \ \  f\in \mathbb{F}_q((1/t)),\ \  x\in \Lie(G)(K_{\infty}).
		\]
		We set  $W_{G}(K_{\infty}):=\frac{\Lie(G)(K_{\infty})}{(\rd_{\varphi}(t)-\theta\Id_{d})\Lie(G)(K_{\infty})}$. It has an $\bA$-module structure induced from the $\bA$-action on $\Lie(G)(K_{\infty})$. We also define $W_{G}(A)$ to be the $\bA$-submodule of $W_{G}(K_{\infty})$ consisting of equivalence classes with coefficients in $A$. 
		
		Extending the definition of a monic polynomial in $A$, we say that an element $\sum_{i\geq i_0}c_i\theta^{-i}\in K_{\infty}$ is \textit{monic}, if the leading coefficient $c_{i_0}\in \mathbb{F}_q^{\times}$ is equal to 1. Let $\proj:\Lie(G)(K_{\infty})\to W_{G}(K_{\infty})$ be the natural projection. 
		
		Let $n$ be a positive integer. Using \cite[Thm. A]{Mau24}, \cite[Thm. 4.4]{ADTR}, and \cite[Cor. 4.5]{ADTR} for our abelian $t$-modules $\mathcal{E}_{n,A}$ and $G_{n,A}$, one obtains the following.
		\begin{theorem} \label{T:ThmANDTR} Let $G$ be either $\mathcal{E}_{n,A}$  or $G_{n,A}$. Then, the following statements hold.
			\begin{itemize}
				\item[(i)] There exists an $\mathbb{F}_q((1/t))$-vector subspace $\mathcal{Z}$ of $\Lie(G)(K_{\infty})$ which is free of rank $r$ and isomorphic to $W_{G}(K_{\infty})$ via the natural projection $\proj$.
				\item[(ii)] The intersection $Z:=U(G/A)\cap \mathcal{Z}$ is a free $\bA$-module of rank $r$. In particular, for any $\bA$-basis $\{g_1,\dots,g_r\}$ of $Z$, we have $\Exp_{G}(g_i)\in G(A)$ for any $1\leq i \leq r$.
				\item[(iii)] There exists a non-zero element $a\in K$ such that 
				\[
				a\wedge_{A}^{r}\proj(Z)=L(G/A)\wedge_{A}^{r}W_{G}(A)
				\] 
				where $\wedge_{A}^{r}\proj(Z)$ ($\wedge_{A}^{r}W_{G}(A)$ respectively) is the monic generator of the $r$-th exterior power of the free $\bA$-module $\proj(Z)$ ($W_{G}(A)$ respectively). 
			\end{itemize}
		\end{theorem}
		\begin{remark} We refer the interested reader to the proof of \cite[Prop. 4.3]{ADTR} for the explicit construction of the $\mathbb{F}_q((1/t))$-vector space $\mathcal{Z}$. 
		\end{remark}
		Note that, for $n\geq 1$, $a\in A$ and $x_{rn},\dots,x_{rn+r-1}\in K_{\infty}$, the $\bA$-module structure on $W_{\mathcal{E}_{n,A}}(K_{\infty})$ can be explicitly given by
		\begin{multline*}
		a(t)\cdot(0,\dots,0,x_{rn},\dots,x_{rn+r-1})^{\tr}+ (\rd_{\varphi_{n,A}}(t)-\theta\Id_{rn+r-1})\Lie(\mathcal{E}_{n,A})(K_{\infty})\\=(0,\dots,0,ax_{rn},\dots,ax_{rn+r-1})^{\tr}
		+ (\rd_{\varphi_{n,A}}(t)-\theta\Id_{rn+r-1})\Lie(\mathcal{E}_{n,A})(K_{\infty}).
		\end{multline*}

		If we write $Z:=\oplus_{i=1}^r\bA g_i$ for some $\bA$-basis $\{g_1,\dots,g_r\}$ where $g_i=[g_{i,1},\dots,g_{i,rn+r-1}]^{\tr}\in \Lie(\mathcal{E}_{n,A})(K_{\infty})$ for $1\leq i \leq r$, then $\proj(Z)=\oplus_{i=1}^r\bA \bar{g}_i$ where 
		\[
		\bar{g}_i:=\begin{bmatrix}
		0,
		\dots,
		0,
		g_{i,rn},
		\dots,
		g_{i,rn+r-1}
		\end{bmatrix}^{\tr}+(\rd_{\varphi_{n,A}}(t)-\theta\Id_{rn+r-1})\Lie(\mathcal{E}_{n,A})(K_{\infty}).
		\] 
		Thus, we see that $\wedge_{A}^r\proj(Z)=c\det(\mathcal{R})$ where $\mathcal{R}$ is the matrix
		\begin{equation*}\label{E:matrix}
		\mathcal{R}:=\begin{bmatrix}
		g_{1,rn}&\dots & g_{r,rn}\\
		\vdots& & \vdots \\
		g_{1,rn+r-1}&\dots & g_{r,rn+r-1}
		\end{bmatrix}\in \Mat_{r}(K_{\infty})
		\end{equation*}
		and $c\in \mathbb{F}_q^{\times}$ so that $c\det(\mathcal{R})$ is monic in $K_{\infty}$.
		In a similar way, one can easily calculate that 
		\begin{equation*}\label{E:matrix11}
		\wedge_{A}^rW_{\mathcal{E}_{n,A}}(A)=1.
		\end{equation*}
		
		Recall that $\gamma=((-1)^{r-1}a_r^{-1})^{1/(q-1)}\in \oK$ and for any $n\geq 0$, observe the following identity in $\Mat_{rn+r-1}(K)[[\tau]]]$:
		\begin{equation}\label{E:Ide}
		\Exp_{\mathcal{E}_{n,A}}=(\mathfrak{a}\gamma)^{-1}\Exp_{\mathcal{E}_n} \mathfrak{a}\gamma.
		\end{equation}

 Recall the endomorphism ring $\End(\mathcal{E}_{n,A})$ ($\End(G_{n,A})$ resp.) of $\mathcal{E}_{n,A}$ ($G_{n,A}$ resp.) from Definition \ref{D:def}(iii). Since, by \cite[Prop. 3.7, Prop. 4.6]{GN24}, $\End(\mathcal{E}_{n,A})$ ($\End(G_{n,A})$ resp.) is an integral domain, one can consider the fraction field $K_n$ ($K_n^{\tens}$ resp.) of $\End(\mathcal{E}_{n,A})$ ($\End(G_{n,A})$ resp.). Before stating the proof of Theorem \ref{T:3}, we say that  the elements $\bsz_1,\dots,\bsz_k\in \Lie(\mathcal{E}_{n,A})(\CC_{\infty})$ are \textit{linearly independent over $K_n$} if whenever $\rd \rP_1\bsz_1+\dots+\rd \rP_k\bsz_k=0$ for some $\rP_1,\dots,\rP_k\in K_n$, we have $\rP_1=\dots=\rP_k=0$. The analogous definition can be made for \textit{linearly independency over $K_n^{\tens}$} similarly. 
		
		\begin{proof}[{Proof of Theorem \ref{T:3}}] When $n=0$, we know by the class number formula of Fang \cite[Thm. 1.10]{Fang} for abelian $t$-modules that  $L(\mathcal{E}_{0,A}/A)$ may be written as a product of a non-zero polynomial in $A$ and the determinant $\mathfrak{D}$ of a matrix consisting of the entries of some $\bsy_1',\dots,\bsy_{r-1}'\in \Lie(\mathcal{E}_{0,A})(K_{\infty})$ so that $\Exp_{\mathcal{E}_{0,A}}(\bsy_i')\in \mathcal{E}_{0,A}(A)$ for each $1\leq i \leq r-1$. Up to rearrangement if necessary, after setting $\bsy_i=[y_{i,1},\dots,y_{i,r-1}]^{\tr}$, let $\{\bsy_{1}, \dots, \bsy_{k}\}$ be the maximal linearly independent subset of $\{
			\bsy_{1}', \dots, \bsy_{r-1}'\}$ over $K_{0}$, where $k \leq r-1$. Then one can write
			\[
			\mathfrak{D} = \widetilde{\bff}(y_{1,1}, \dots, y_{1,r-1}, \dots, \dots, y_{k,1}, \dots, y_{k,r-1}),
			\]
			for some non-constant polynomial $\widetilde{\bff} \in \oK[X_1,\dots,X_{(r-1)k}]$. Since $L(\mathcal{E}_{0,A}/A)$ is non-zero, we conclude the proof of this case by using \eqref{E:Ide} and \cite[Thm. 1.1(i)]{GN24}. Now let $n\in \ZZ_{\geq 1}$ and $\mathcal{E}_{n,A}$ be the abelian $t$-module defined as in \S3.2. By Theorem \ref{T:ThmANDTR}(iii), we have that 
			\[
			L(\mathcal{E}_{n,A}/A)=a\det(\mathcal{R})
			\]
			for some $a\in K^{\times}$. Since $L(\mathcal{E}_{n,A}/A)$ is non-zero, using the same reasoning above to show the transcendence of the determinant in the $n=0$ case, by \eqref{E:Ide} and \cite[Thm. 1.1(ii)]{GN24}, we obtain that $L(\mathcal{E}_{n,A}/A)\not \in \oK$. 
		\end{proof}
		
		\begin{proof}[{Proof of Theorem \ref{T:2}}] When $n=0$, note that the corresponding abelian $t$-module to $G_{0}$ is a Drinfeld $\bA$-module defined over $A$. Thus, \cite[Cor. 4.6]{ChangEl-GuindyPapanikolas} implies the desired result.  We now prove the case $n\in \ZZ_{\geq 1}$. Since the matrix $N$ given in \eqref{E:matrices} has only zeros in its last $r$ rows,  repeating the same argument above by using Theorem \ref{T:ThmANDTR}, one can see that there exist elements $\tilde{g}_1,\dots,\tilde{g}_r\in \Lie(G_{n})(K_{\infty})$ satisfying $\Exp_{G_{n}}(\tilde{g}_\ell)\in G_n(A)$ for each $1\leq \ell \leq r$ and a non-zero $\tilde{a}\in K$ so that 
			\[
			L(G_{n}/A)=\tilde{a}\det(\mathcal{R}')
			\]
			where $\tilde{g}_{\ell}=[\tilde{g}_{\ell,1},\dots,\tilde{g}_{\ell,rn+1}]^{\tr}$ and $\mathcal{R}':=(\tilde{g}_{i,r(n-1)+1+j})_{ij}\in \Mat_{r}(K_{\infty})$. Up to rearrangement if necessary, let $\{g_{1}, \dots, g_{k}\}$ be the maximal linearly independent subset of $\{
			\tilde{g}_{1}, \dots, \tilde{g}_{r}\}$ over $K_{n}^{\tens}$ where $k \leq r$. Then one can write
			\begin{equation}\label{E:max}
			\det(\mathcal{R}') = \widetilde{\bff}(g_{1, r(n-1)+2}, \dots, g_{1, rn+1}, \dots, g_{k, r(n-1)+2}, \dots, g_{k, rn+1}),
			\end{equation}
			for some non-constant polynomial $\widetilde{\bff} \in \oK[X_1,\dots,X_{rk}]$.  Thus, by \eqref{E:max} and \cite[Thm. 1.3]{GN24}, $\det(\mathcal{R}')\in K_{\infty}$ is either zero or transcendental over $\oK$. Since $\tilde{a}$  and $L(G_{n}/A)$ are non-zero, we conclude that $L(G_{n}/A)\not \in \oK$. 
		\end{proof}
		
		The following corollary, giving a positive answer to \cite[Problem 4.1]{ADTR2020} in the case of the tensor product of Drinfeld $\bA$-modules of rank $r$ defined over $A$ and their $(r-1)$-st exterior powers with Carlitz tensor powers, immediately follows from  Theorem \ref{T:3} and Theorem \ref{T:2}.
		
		\begin{corollary}\label{C:Taelman}
			Let $G_{0,A}$ be a Drinfeld A-module of rank $r$ defined over $A$ and for any nonnegative integer $n$, let $G_{0,A} \otimes  C^{\otimes n}$ be the abelian $t$-module defined in \S2.5.
            
            \begin{itemize}
				\item[(i)] The Taelman $L$-value $L(G_{0,A} \otimes  C^{\otimes n}/A)$ of $G_{0,A} \otimes C^{\otimes n}$
				is transcendental over $K$.
				\item[(ii)] Assume that $r\geq  2$. The Taelman $L$-value $L((\wedge^{r-1}G_{0,A})\otimes C^{\otimes n}/A)$ of $(\wedge^{r-1}G_{0,A}) \otimes C^{\otimes n}$
				is
				transcendental over $K$.
			\end{itemize}
		\end{corollary}

\end{document}